\documentclass[a4paper, 10pt]{article}

\usepackage[latin1]{inputenc}
\usepackage{subfigure}
\usepackage{amsmath}
\usepackage{amssymb}
\usepackage{amsthm}
\usepackage{amsfonts}
\usepackage[cyr]{aeguill}
\usepackage{xspace}
\usepackage[english]{babel}
\usepackage{graphicx}
\usepackage{float}
\usepackage{psfrag,floatflt,hyperref}
\usepackage{verbatim}
\usepackage{geometry}
\usepackage{multirow}
\usepackage[retainorgcmds]{IEEEtrantools}
\usepackage{color}
\newtheorem{theorem}{Theorem}[section]
\newtheorem{lemma}[theorem]{Lemma}
\newtheorem{proposition}[theorem]{Proposition}

\newenvironment{remark}[1][Remark.]{\begin{trivlist}
\item[\hskip \labelsep {\bfseries #1}]}{\end{trivlist}}
\newenvironment{example}[1][Example.]{\begin{trivlist}
\item[\hskip \labelsep {\bfseries #1}]}{\end{trivlist}}

\title{Asymptotics for the number of spanning trees in circulant graphs and degenerating $d$-dimensional discrete tori\footnote{The author acknowledges support from the Swiss NSF grant $200021\_132528/1$.}}
\author{Justine Louis}
\date{1 December 2014}

\DeclareMathOperator{\argcosh}{argcosh}
\newcommand*{\defeq}{\mathrel{\vcenter{\baselineskip0.5ex \lineskiplimit0pt
                     \hbox{\scriptsize.}\hbox{\scriptsize.}}}=}

\begin{document}
        \maketitle
\begin{flushright}
\end{flushright}

\begin{abstract}
In this paper we obtain precise asymptotics for certain families of graphs, namely circulant graphs and degenerating discrete tori. The asymptotics contain interesting constants from number theory among which some can be interpreted as corresponding values for continuous limiting objects.
We answer one question formulated in a paper from Atajan, Yong and Inaba in \cite{atajan2006further} and formulate a conjecture in relation to the paper from Zhang, Yong and Golin \cite{zhang2005chebyshev}.
A crucial ingredient in the proof is to use the matrix tree theorem and express the combinatorial Laplacian determinant in terms of Bessel functions. A non-standard Poisson summation formula and limiting properties of theta functions are then used to evaluate the asymptotics.
\end{abstract}

\begin{flushleft}
\end{flushleft}

\section{Introduction}
The number of spanning trees of a finite graph is an interesting invariant which has many applications in different fields such as network reliability (for example see \cite{colbourn1987combinatorics}), statistical physics \cite{mcdonald2012potts}, designing electrical circuits; for more applications see \cite{cvetkovic1980spectra}. In $1847$ Kirchhoff established the matrix tree theorem \cite{kirchhoff1847ueber} which relates the number of spanning trees $\tau(G)$ in a graph $G$ with $\lvert V(G)\rvert$ vertices to the determinant of the combinatorial Laplacian on $G$ by the following relation
\begin{equation*}
\tau(G)=\frac{1}{\lvert V(G)\rvert}\textrm{det}^\ast\Delta
\end{equation*}
where $\textrm{det}^\ast\Delta$ is the product of the non-zero eigenvalues of the Laplacian on $G$.

One type of graphs, so-called circulant graphs, also known as loop networks, has been much studied. Let $1\leqslant\gamma_1\leqslant\cdots\leqslant\gamma_d\leqslant\lfloor n/2\rfloor$ be positive integers. A circulant graph $C^{\gamma_1,\ldots,\gamma_d}_n$ is the $2d$-regular graph with $n$ vertices labelled $0,1,\ldots,n-1$ such that each vertex $v\in\mathbb{Z}/n\mathbb{Z}$ is connected to $v\pm\gamma_i$ mod $n$ for all $i\in\{1,\ldots,d\}$. Figure \ref{circ} illustrates two examples. The problem of computing the number of spanning trees in these graphs can be approached in several ways. One of the first results, proved by Kleitman and Golden \cite{kleitman1975counting}, see also \cite{baron1985number} and \cite{yong1997numbers}, states that $\tau(C_n^{1,2})=nF_n^2$, where $F_n$ are the Fibonacci numbers. Boesch and Prodinger \cite{boesch1986spanning} computed the number of spanning trees for different classes of graphs with algebraic techniques using Chebyshev polynomials. Zhang, Yong and Golin \cite{
golin2002further,
zhang2005chebyshev} used this technique for circulant graphs. The same authors showed in \cite{zhang1999number} that the number of spanning trees in circulant graphs with fixed generators satisfies a recurrence relation, that is $\tau(C_n^{\gamma_1,\ldots,\gamma_d})=na_n^2$ where $a_n$ satisfies a recurrence relation of order $2^{\gamma_d-1}$. This was also proved combinatorially later by Golin and Leung in \cite{golin2005unhooking}. They extended their method to circulant graphs with non-fixed generators in \cite{golin2005counting}. In \cite{atajan2006further}, Atajan, Yong and Inaba improved the order of the recurrence relation for $a_n$ and found the asymptotic behaviour of $a_n$, \textit{i.e.} $a_n\sim c\phi^n$, where $c$ and $\phi$ are constants which are obtained from the smallest modulus root of the generating function of $a_n$. They again improved this in \cite{atajan2010efficient} by finding an efficient way of solving the recurrence relation of $a_n$.
\begin{figure}[!h]
\label{circ}
\centering
\includegraphics[width=12cm]{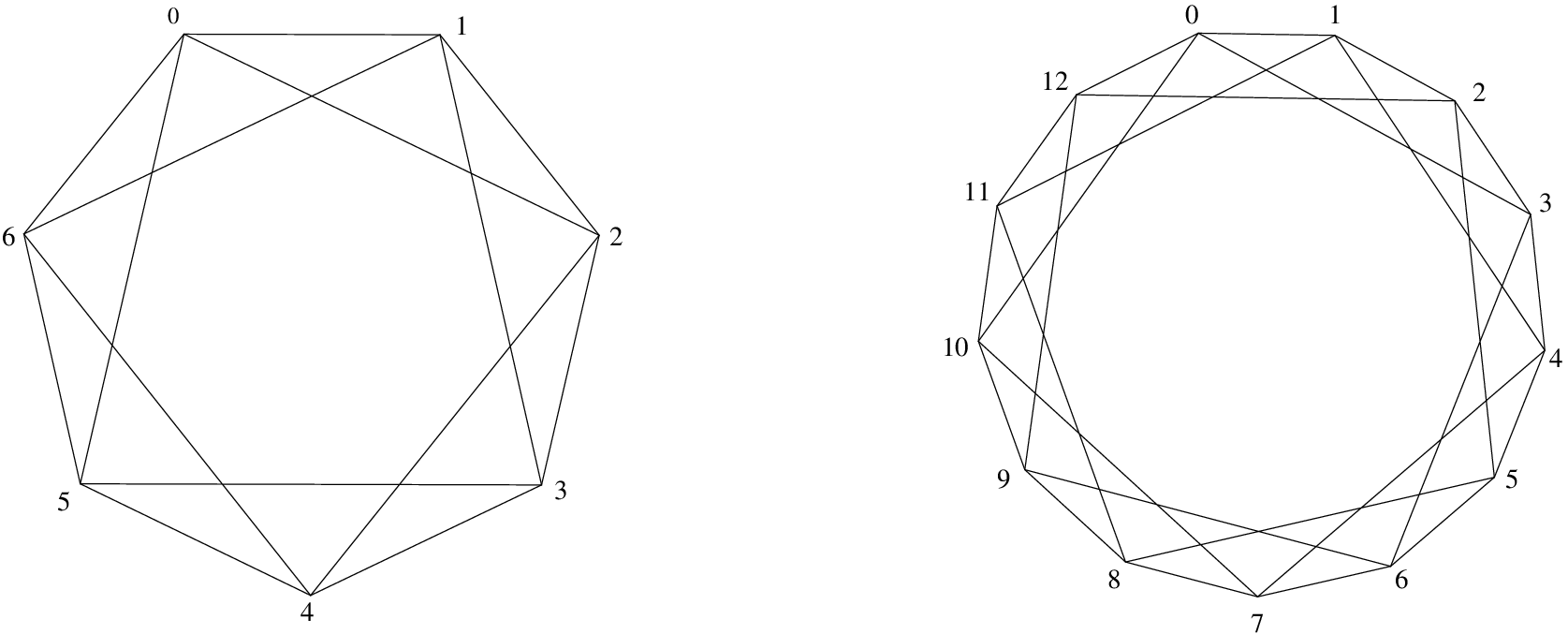}
\caption{The circulant graphs $C_7^{1,2}$ and $C_{13}^{1,3}$.}
\end{figure}

In this work we are interested in studying the asymptotic behaviour of the number of spanning trees in circulant graphs with fixed generators and in $d$-dimensional discrete tori.
This will be done by extending the work of Chinta, Jorgenson and Karlsson in \cite{chinta2010zeta} and \cite{chinta2011complexity} to these cases. In their papers, the authors developed a technique to compute the asymptotic behaviour of spectral determinants of sequences of discrete tori $\mathbb{Z}^d/\Lambda_n\mathbb{Z}^d$ where $\Lambda_n$ is a $d\times d$ integer matrix such that $\textrm{det}\ \Lambda_n\rightarrow\infty$ and $\Lambda_n/(\textrm{det}\ \Lambda_n)^{1/d}\rightarrow A\in SL_d(\mathbb{R})$ as $n\rightarrow\infty$. The two families of graphs which will be considered here do not satisfy this condition. An important ingredient is the theta inversion formula (see Proposition \ref{thetainv} below) which relates the eigenvalues of the combinatorial Laplacian to the modified $I$-Bessel functions. The method then consists in studying the asymptotics of integrals involving these Bessel functions. In the first part of this work we apply it to the case of circulant graphs with fixed generators. We will 
prove the following theorem:
\begin{theorem}
\label{ThCirc}
Let $C_n^\Gamma$ be a circulant graph with $n$ vertices and $d$ generators independent of $n$ given by $\Gamma\defeq\{1,\gamma_1,\ldots,\gamma_{d-1}\}$, such that $1\leqslant\gamma_1\leqslant\cdots\leqslant\gamma_{d-1}\leqslant\lfloor\frac{n}{2}\rfloor$, and let $\textnormal{det}^\ast\Delta_{C_n^\Gamma}$ be the product of the non-zero eigenvalues of the Laplacian on $C_n^\Gamma$. Then as $n\rightarrow\infty$
\begin{equation*}
\log\textnormal{det}^\ast\Delta_{C_n^\Gamma}=n\int_0^\infty(e^{-t}-e^{-2dt}I_0^\Gamma(2t,\ldots,2t))\frac{dt}{t}+2\log n-\log c_\Gamma+o(1)
\end{equation*}
where $c_\Gamma=1+\sum_{i=1}^{d-1}\gamma_i^2$ and
\begin{equation*}
I_0^\Gamma(2t,\ldots,2t)=\frac{1}{2\pi}\int_{-\pi}^\pi e^{2t(\cos w+\sum_{i=1}^{d-1}\cos\gamma_iw)}dw
\end{equation*}
is the $d$-dimensional modified $I$-Bessel function of order zero.
\end{theorem}
The function $I_0^\Gamma$ appearing in the lead term is a generalization of the $2$-dimensional $J$-Bessel function in \cite{korsch2006two} and will be defined in section \ref{ddimB}.\\
Theorem \ref{ThCirc} can be compared to Lemma $2$ of Golin, Yong and Zhang in \cite{golin2010asymptotic} where they find the lead term of the asymptotic number of spanning trees. With our method we derived also the second term of the asymptotic. These are consistent with numerics given in \cite{zhang1999number,atajan2006further} by these authors. In particular, this answers one of their open problems stated in the conclusion of \cite{atajan2006further} that asks whether we can find out the exact value of the asymptotic constants. Indeed we show that
\begin{equation*}
c^2=\frac{1}{c_\Gamma}.
\end{equation*}

Let $\Lambda_n$ be a $d\times d$ invertible diagonal integer matrix. In the second part of this work we extend the method used in \cite{chinta2010zeta} to study, in two different cases, the asymptotic behaviour of spectral determinants of a sequence of $d$-dimensional degenerating discrete tori, that is, the Cayley graphs of the groups $\mathbb{Z}^d/\Lambda_n\mathbb{Z}^d$ with respect to the generators corresponding to the standard basis vectors of $\mathbb{Z}^d$. In the first case, it is degenerating in the sense that $d-p$ sides of the torus are tending to infinity at the same rate while $p$ sides tend to infinity sublinearly with respect to the $d-p$ sides. More precisely, let $\alpha_i$, $i=1,\ldots,p$, and $\beta_i$, $i=1,\ldots,d-p$, be positive non-zero integers and $a_n$ be a sequence of positive integers, the matrix $\Lambda_n$ considered is then given by $\Lambda_n=\textrm{diag}(\alpha_1a_n,\ldots,\alpha_pa_n,\beta_1n,\ldots,\beta_{d-p}n)$. In the first case, $a_n$ goes to infinity 
sublinearly with respect to $n$, that is
\begin{equation*}
\frac{a_n}{n}\rightarrow0\textnormal{ as }n\rightarrow\infty.
\end{equation*}
In the second case, the size of the $p$ sides of the torus stay constant ($a_n=1$ for all $n$) while the $d-p$ other sides go to infinity at the same rate. The matrix considered, denoted by $\Lambda_n^0$, is therefore given by $\Lambda_n^0=\textrm{diag}(\alpha_1,\ldots,\alpha_p,\beta_1n,\ldots,\beta_{d-p}n)$. Figure \ref{dtorus} illustrates an example.
\begin{figure}[!h]
\centering
\includegraphics[scale=0.8]{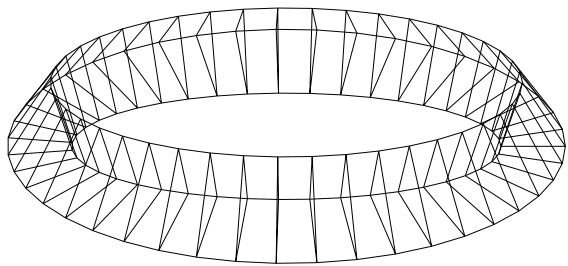}
\caption{The discrete torus $\mathbb{Z}/n\mathbb{Z}\times\mathbb{Z}/\lfloor\log n\rfloor\mathbb{Z}$ with $n=43$.}
\label{dtorus}
\end{figure}
\\
Let $M$ be an $r\times r$ invertible matrix. We define the spectral or Epstein zeta function associated to the real torus $\mathbb{R}^r/M\mathbb{Z}^r$, for $\textrm{Re}(s)>r/2$ by
\begin{equation*}
\zeta_{\mathbb{R}^r/M\mathbb{Z}^r}(s)=\frac{1}{(2\pi)^{2s}}\sum_{k\in\mathbb{Z}^r\backslash\{0\}}(k^TM^{-1}k)^{-s}.
\end{equation*}
It has an analytic continuation to the whole complex plane except for a simple pole at $s=r/2$. The regularized determinant of the Laplacian on the real torus $\mathbb{R}^r/M\mathbb{Z}^r$ is then defined through the spectral zeta function evaluated at $s=0$ by
\begin{equation*}
\log\textrm{det}^\ast\Delta_{\mathbb{R}^r/M\mathbb{Z}^r}=-\zeta'_{\mathbb{R}^r/M\mathbb{Z}^r}(0).
\end{equation*}
From now on, the matrices $A$, $B$ and $\Lambda$ will denote:
\begin{equation*}
A=\textnormal{diag}(\alpha_1,\ldots,\alpha_p),\quad B=\textnormal{diag}(\beta_1,\ldots,\beta_{d-p}),\quad\Lambda=\textnormal{diag}(\alpha_1,\ldots,\alpha_p,\beta_1,\ldots,\beta_{d-p}).
\end{equation*}
We will prove the two following theorems.
\begin{theorem}
\label{ThDegT}
Let $\textnormal{det}^\ast\Delta_{\mathbb{Z}^d/\Lambda_n\mathbb{Z}^d}$ be the product of the non-zero eigenvalues of the Laplacian on the discrete torus $\mathbb{Z}^d/\Lambda_n\mathbb{Z}^d$. Then as $n\rightarrow\infty$
\begin{align*}
\log\textnormal{det}^\ast\Delta_{\mathbb{Z}^d/\Lambda_n\mathbb{Z}^d}&=n^{d-p}a_n^p\det(\Lambda)c_d-\left(\frac{n}{a_n}\right)^{d-p}\left(\det(\Lambda)(4\pi)^{d/2}\Gamma(d/2)\zeta_{\mathbb{R}^p/A^{-1}\mathbb{Z}^p}(d/2)+o(1)\right)
\end{align*}
where $c_d$ is the following integral
\begin{equation*}
c_d=\int_0^\infty\left(e^{-t}-e^{-2dt}I_0(2t)^d\right)\frac{dt}{t}.
\end{equation*}
\end{theorem}
We recall the special values for the gamma function for odd $d$, $\Gamma(d/2)=(d-2)!!\sqrt{\pi}/2^{(d-1)/2}$, and for even $d$, $\Gamma(d/2)=(d/2-1)!$.
\begin{theorem}
\label{ThDegTcst}
Let $\textnormal{det}^\ast\Delta_{\mathbb{Z}^d/\Lambda_n^0\mathbb{Z}^d}$ be the product of the non-zero eigenvalues of the Laplacian on the discrete torus $\mathbb{Z}^d/\Lambda_n^0\mathbb{Z}^d$. Then as $n\rightarrow\infty$
\begin{align*}
\log\textnormal{det}^\ast\Delta_{\mathbb{Z}^d/\Lambda_n^0\mathbb{Z}^d}&=n^{d-p}\det(B)\sum_{j=0}^{\det(A)-1}\int_0^\infty\left(e^{-t}-I_0(2t)^{d-p}e^{-(2(d-p)+\lambda_j)t}\right)\frac{dt}{t}\\
&\ \ \ +2\log{n}-\zeta'_{\mathbb{R}^{d-p}/B\mathbb{Z}^{d-p}}(0)+o(1)
\end{align*}
where $\lambda_j$, $j=0,1,\ldots,\det(A)-1$, are the eigenvalues of the Laplacian on $\mathbb{Z}^p/A\mathbb{Z}^p$ given by
\begin{equation*}
\{\lambda_j\}_j=\{2p-2\sum_{i=1}^p\cos(2\pi j_i/\alpha_i):j_i=0,1,\ldots,\alpha_i-1,\textnormal{ for }i=1,\ldots,p\}.
\end{equation*}
\end{theorem}
The second term in Theorem \ref{ThDegT} is new in the asymptotic development which comes from the degeneration. In Theorem \ref{ThDegTcst} the terms are similar to the usual ones appearing in the asymptotic behaviour of spectral determinants (see \cite{chinta2010zeta} and \cite{chinta2011complexity}). As mentioned above the last term is the logarithm of the spectral determinant of the Laplacian on the real torus $\mathbb{R}^{d-p}/B\mathbb{Z}^{d-p}$ where $p$ dimensions are lost because of the degeneration of the sequence of tori. Indeed one can rescale the discrete torus by dividing the number of vertices per dimension by $n$. Therefore the $d$-dimensional sequence of discrete tori converges in some sense to the $(d-p)$-dimensional real torus $\mathbb{R}^{d-p}/B\mathbb{Z}^{d-p}$.
\begin{example}
To illustrate Theorem \ref{ThDegT} we consider the graphs $\mathbb{Z}^3/\Lambda_n\mathbb{Z}^3$ where
\begin{equation*}
\Lambda_n=\left(\begin{array}{ccc}\lfloor\log n\rfloor&0&0\\0&n&0\\0&0&n\end{array}\right).
\end{equation*}
Then as $n\rightarrow\infty$
\begin{align*}
\log\textnormal{det}^\ast\Delta_{\mathbb{Z}^3/\Lambda_n\mathbb{Z}^3}&=c_3n^2\lfloor\log n\rfloor-\left(\frac{n}{\lfloor\log n\rfloor}\right)^2\left(\frac{1}{\pi}\zeta(3)+o(1)\right).
\end{align*}
\end{example}
This work is structured as follows. In subsection \ref{Laplacian} we define the combinatorial Laplacian, and then the spectral zeta function and the theta function in subsection \ref{zeta}. In subsection \ref{IB} we recall some results on modified $I$-Bessel functions and in the next subsection we define the $d$-dimensional modified $I$-Bessel function which will be used in the computation of the asymptotics for the circulant graphs. In the two next subsections we recall some upper bounds on modified $I$-Bessel functions and briefly describe the method used in \cite{chinta2010zeta}. In section \ref{asc} we show Theorem \ref{ThCirc} and compare the results with other papers. In section \ref{asdt} we treat the case of the degenerating sequence of tori, show Theorems \ref{ThDegT} and \ref{ThDegTcst} and give some examples. In the last section we formulate a conjecture on the number of spanning trees in $C_{5n}^{1,n}$, for $n\geqslant2$.
\par\vspace{\baselineskip}
\noindent
\textbf{Acknowledgements:} The author gratefully thanks Anders Karlsson for valuable discussions, comments and a careful reading of the manuscript. The author also thanks Fabien Friedli for useful discussions. The author is grateful to the referees for useful comments.
\section{Preliminary results}
\subsection{Laplacians}
\label{Laplacian}
We define a $d$-dimensional discrete torus to be the quotient $\mathbb{Z}^d/M\mathbb{Z}^d$ where $M\in GL_d(\mathbb{Z})$ and a $d$-dimensional real torus by the quotient $\mathbb{R}^d/C\mathbb{Z}^d$ where $C\in GL_d(\mathbb{R})$. Let $C^\ast$ be the matrix generating the dual lattice of $C\mathbb{Z}^d$ defined by
\begin{equation*}
C^\ast\mathbb{Z}^d=\{y\in\mathbb{R}^d\vert\langle x,y\rangle\in\mathbb{Z},\ \forall x\in C\mathbb{Z}^d\}
\end{equation*}
where $\langle\cdot,\cdot\rangle$ is the usual inner product, which satisfies the two following conditions:
\begin{align*}
&\circ\textrm{span}(C)=\textrm{span}(C^\ast)\\
&\circ C^TC^\ast=1.
\end{align*}
The eigenfunctions of the Laplace-Beltrami operator $-\sum_{j=1}^d\partial^2/\partial x_j^2$ on the real torus are given by $\phi(x)=\exp(2\pi i\langle\mu,x\rangle)$, for some $\mu\in\mathbb{R}^d$, with the condition that the opposite sides of the parallelogram generated by $C\mathbb{Z}^d$ are identified. So for all $x\in\mathbb{R}^d$ we have $\phi(x+C\mathbb{Z}^d)=\phi(x)$. Hence $\exp(2\pi i\langle\mu,C\mathbb{Z}^d\rangle)=1$ and therefore $\langle\mu,C\mathbb{Z}^d\rangle\in\mathbb{Z}$ if and only if $\mu=C^\ast m$ for $m\in\mathbb{Z}^d$. It follows that the eigenvalues are given by
\begin{equation}
\label{evRT}
\lambda_m=(2\pi)^2\mu^T\mu=(2\pi)^2\lVert C^\ast m\rVert^2\textrm{ with }m\in\mathbb{Z}^d.
\end{equation}
Let $V(\mathbb{Z}^d/M\mathbb{Z}^d)$ be the set of vertices of the torus $\mathbb{Z}^d/M\mathbb{Z}^d$ and $f:V(\mathbb{Z}^d/M\mathbb{Z}^d)\rightarrow\mathbb{C}$. The combinatorial Laplacian on $\mathbb{Z}^d/M\mathbb{Z}^d$ is defined by
\begin{equation*}
\Delta_{\mathbb{Z}^d/M\mathbb{Z}^d}f(x)=\sum_{y\sim x}(f(x)-f(y))
\end{equation*}
where the sum is over the vertices adjacent to $x$.\\
Recall Proposition $5$ of \cite{chinta2011complexity}:
\begin{proposition}
\label{thetainv}
Let $\lambda_v$, with $v\in M^\ast\mathbb{Z}^d/\mathbb{Z}^d$, be the eigenvalues of $\Delta_{\mathbb{Z}^d/M\mathbb{Z}^d}$. The following formula holds for $t\in\mathbb{R}_{\geqslant0}$
\begin{equation*}
\lvert\det(M)\rvert\sum_{y\in M\mathbb{Z}^d}e^{-2dt}I_{y_1}(2t)\ldots I_{y_d}(2t)=\sum_{v\in M^\ast\mathbb{Z}^d/\mathbb{Z}^d}e^{-t\lambda_v}
\end{equation*}
where $I_{y_i}$ is the modified $I$-Bessel function of order $y_i$.
\end{proposition}
\subsection{Spectral zeta function and theta function}
\label{zeta}
In this section we define the spectral zeta function and the theta function and give the relations that will enable us to compute the asymptotics in sections \ref{asc} and \ref{asdt}.\\
Let $\{\lambda_j\}_{j\geqslant0}$ be the eigenvalues of the combinatorial Laplacian, respectively the Laplace-Beltrami operator, on a discrete torus, respectively a real torus, T, with $\lambda_0=0$. The associated theta function on $T$ is defined by
\begin{equation}
\label{theta}
\sum_je^{-\lambda_jt}.
\end{equation}
It will be denoted by $\theta_T(t)$ when $T$ denotes a discrete torus and by $\Theta_T(t)$ when $T$ denotes a real torus. The relation in Proposition \ref{thetainv} is then called the theta inversion formula on $\mathbb{Z}^d/M\mathbb{Z}^d$. The associated spectral zeta function on a real torus $T$ is defined for $\textrm{Re}(s)>d/2$ by
\begin{equation*}
\zeta_T(s)=\sum_{j\neq0}\frac{1}{\lambda_j^s}.
\end{equation*}
It is related to the theta function through the Mellin transform:
\begin{equation*}
\zeta_T(s)=\frac{1}{\Gamma(s)}\int_0^\infty(\Theta_T(t)-1)t^s\frac{dt}{t}
\end{equation*}
where the $-1$ in the integral comes from the fact that the zero eigenvalue is kept in the definition of the theta function, and where $\Gamma(s)=\int_0^\infty e^{-t}t^sdt/t$ is the gamma function.\\
Let $M\in GL_d(\mathbb{R})$ be a matrix. By splitting the above integral one can show that the zeta function admits a meromorphic continuation to $s\in\mathbb{C}$ (see section $2.6$ in \cite{chinta2010zeta}). By differentiating $\zeta_{\mathbb{R}^d/M\mathbb{Z}^d}$ and evaluating at $s=0$, one has
\begin{align}
\label{zeta'(0)}
\zeta'_{\mathbb{R}^d/M\mathbb{Z}^d}(0)&=\int_0^1(\Theta_{\mathbb{R}^d/M\mathbb{Z}^d}(t)-\lvert\textrm{det}(M)\rvert(4\pi t)^{-d/2})\frac{dt}{t}+\Gamma'(1)\nonumber\\
&\ \ \ -\frac{2}{d}\lvert\textrm{det}(M)\rvert(4\pi)^{-d/2}+\int_1^\infty(\Theta_{\mathbb{R}^d/M\mathbb{Z}^d}(t)-1)\frac{dt}{t}.
\end{align}
In section \ref{asc} a limiting torus will be the circle $S^1=\mathbb{R}/\mathbb{Z}$. In this case it is convenient to split the integral at $c_\Gamma$. The spectral zeta function is defined for $\textrm{Re}(s)>1/2$:
\begin{align*}
\zeta_{S^1}(s)&=\frac{1}{\Gamma(s)}\int_0^\infty(\Theta_{S^1}(t)-1)t^s\frac{dt}{t}\\
&=\frac{1}{\Gamma(s)}\int_0^{c_\Gamma}\left(\Theta_{S^1}(t)-\frac{1}{\sqrt{4\pi t}}\right)t^s\frac{dt}{t}+\frac{1}{\Gamma(s)}\int_0^{c_\Gamma}\left(\frac{1}{\sqrt{4\pi t}}-1\right)t^s\frac{dt}{t}\\
&\ \ \ +\frac{1}{\Gamma(s)}\int_{c_\Gamma}^\infty(\Theta_{S^1}(t)-1)t^s\frac{dt}{t}\\
&=\frac{1}{\Gamma(s)}\int_0^{c_\Gamma}\left(\Theta_{S^1}(t)-\frac{1}{\sqrt{4\pi t}}\right)t^s\frac{dt}{t}+\frac{1}{\Gamma(s)}\left(\frac{c_\Gamma^{s-1/2}}{\sqrt{4\pi}(s-1/2)}-\frac{c_\Gamma^s}{s}\right)\\
&\ \ \ +\frac{1}{\Gamma(s)}\int_{c_\Gamma}^\infty(\Theta_{S^1}(t)-1)t^s\frac{dt}{t}.
\end{align*}
This defines a meromorphic continuation of $\zeta_{S^1}$ to the whole complex plane, hence the limit of $\zeta_{S^1}(s)$ at $s=0$ exists. Near $s=0$ the gamma function behaves as $1/\Gamma(s)=s+O(s^2)$. Therefore
\begin{equation}
\label{zetaS1}
\zeta'_{S^1}(0)=\int_0^{c_\Gamma}\left(\Theta_{S^1}(t)-\frac{1}{\sqrt{4\pi t}}\right)\frac{dt}{t}-\frac{1}{\sqrt{\pi c_\Gamma}}-\log c_\Gamma+\Gamma'(1)+\int_{c_\Gamma}^\infty(\Theta_{S^1}(t)-1)\frac{dt}{t}.
\end{equation}
As mentioned in the introduction, we notice that for a real torus $T$ the regularized determinant of the Laplacian, $\textrm{det}^\ast\Delta_T$, is defined by the following identity (for more details see \cite{voros1987spectral}):
\begin{equation*}
\log\textrm{det}^\ast\Delta_T=-\zeta'_T(0).
\end{equation*}
Let $s\in\mathbb{C}$ with $\textrm{Re}(s)>d/2$, and $M=\textrm{diag}(m_1,\ldots,m_d)$ be a positive diagonal matrix. Using (\ref{evRT}), the zeta function can be rewritten as
\begin{equation}
\label{zeta2}
\zeta_{\mathbb{R}^d/M\mathbb{Z}^d}(s)=\frac{1}{(4\pi^2)^s}\sum_{(k_1,\ldots,k_d)\in\mathbb{Z}^d\backslash\{0\}}\frac{1}{\big(\sum_{i=1}^dk_i^2/m_i^2\big)^s}.
\end{equation}
Let $\zeta$ be the Riemann zeta function. In the case of the circle $\mathbb{R}/\beta\mathbb{Z}$ the eigenvalues of the Laplacian are given by $\lambda_j=(2\pi)^2(j/\beta)^2$ for $j\in\mathbb{Z}$, so the spectral zeta function is related to the Riemann zeta function by
\begin{equation*}
\zeta_{\mathbb{R}/\beta\mathbb{Z}}(s)=2(\beta/2\pi)^{2s}\zeta(2s).
\end{equation*}
Using the special values of the Riemann zeta function $\zeta(0)=-1/2$ and $\zeta'(0)=-(1/2)\log(2\pi)$, the derivative evaluated at zero is given by
\begin{equation}
\label{zeta'}
\zeta'_{\mathbb{R}/\beta\mathbb{Z}}(0)=4\log(\beta/2\pi)\zeta(0)+4\zeta'(0)=-2\log\beta.
\end{equation}
In particular for the unit circle $S^1=\mathbb{R}/\mathbb{Z}$, one has
\begin{equation}
\label{zeta'_S1(0)}
\zeta'_{S^1}(0)=0.
\end{equation}
\subsection{\texorpdfstring{Modified \boldmath$I$-Bessel functions\unboldmath\ }{Modified I-Bessel function}}
\label{IB}
Let $I_x$ be the modified $I$-Bessel function of the first kind of index $x$.
For positive integer values of $x$, $I_x(t)$ has the following series representation
\begin{equation}
\label{seriesrepI}
I_x(t)=\sum_{n=0}^\infty\frac{(t/2)^{2n+x}}{n!\Gamma(n+1+x)}
\end{equation}
and the integral representation
\begin{equation*}
I_x(t)=\frac{1}{2\pi}\int_{-\pi}^\pi e^{t\cos\theta}\cos(\theta x)d\theta.
\end{equation*}
For negative values of $x$ we have that $I_x(t)=I_{-x}(t)$ for all $t$.\\
From Theorem $9$ in \cite{karlsson2006heat} which is a special case of Proposition \ref{thetainv}, we have the theta inversion formula on $\mathbb{Z}/m\mathbb{Z}$, that is, for every integer $m>0$ and all $t$,
\begin{equation}
\label{thetainvZ/mZ}
e^{-t}\sum_{k\in\mathbb{Z}}I_{km}(t)=\frac{1}{m}\sum_{j=0}^{m-1}e^{-(1-\cos(2\pi j/m))t}.
\end{equation}
The two following propositions give some results on the asymptotics of the $I$-Bessel function. The first result has been proved in \cite{chinta2010zeta}.
\begin{proposition}
\label{prop4.7}
Let $b(n)$ be a sequence of positive integers parametrized by $n\in\mathbb{N}$ such that $b(n)/n\rightarrow\beta>0$ as $n\rightarrow\infty$. Then for any $t>0$ and non-negative integer $k\geqslant0$, we have
\begin{equation*}
\lim_{n\rightarrow\infty}b(n)e^{-2n^2t}I_{b(n)k}(2n^2t)=\frac{\beta}{\sqrt{4\pi t}}e^{-(\beta k)^2/(4t)}.
\end{equation*}
\end{proposition}
\begin{proposition}
\label{prop4.7bis}
Let $a_n$ be a sequence of positive integers tending to infinity sublinearly with respect to $n$. Then we have that
\begin{equation*}
\label{lim1}
\lim_{n\rightarrow\infty}a_ne^{-2n^2t}\sum_{k\in\mathbb{Z}}I_{a_nk}(2n^2t)=1.
\end{equation*}
\end{proposition}
\begin{proof}
From the theta inversion formula on $\mathbb{Z}$,
\begin{equation*}
a_ne^{-2n^2t}\sum_{k\in\mathbb{Z}}I_{a_nk}(2n^2t)=1+\sum_{j=1}^{a_n-1}e^{-4\sin^2(\pi j/a_n)n^2t}.
\end{equation*}
If $a_n$ is even,
\begin{equation*}
\sum_{j=1}^{a_n-1}e^{-4\sin^2(\pi j/a_n)n^2t}=e^{-4n^2t}+2\sum_{j=1}^{a_n/2-1}e^{-4\sin^2(\pi j/a_n)n^2t}.
\end{equation*}
If $a_n$ is odd,
\begin{equation*}
\sum_{j=1}^{a_n-1}e^{-4\sin^2(\pi j/a_n)n^2t}=2\sum_{j=1}^{(a_n-1)/2}e^{-4\sin^2(\pi j/a_n)n^2t}.
\end{equation*}
Since $e^{-4n^2t}\rightarrow0$ as $n\rightarrow\infty$ both cases behave the same, so we only treat the case where $a_n$ is odd. Using the fact that $\sin x\geqslant x/2$ for all $x\in[0,\pi/2]$, we have
\begin{align*}
\sum_{j=1}^{(a_n-1)/2}e^{-4\sin^2(\pi j/a_n)n^2t}&\leqslant\sum_{j=1}^{(a_n-1)/2}e^{-\pi^2j^2tn^2/a_n^2}\\
&\leqslant\sum_{j=1}^\infty e^{-\pi^2jtn^2/a_n^2}=\frac{1}{e^{\pi^2tn^2/a_n^2}-1}\rightarrow0
\end{align*}
since $n/a_n\rightarrow\infty$ as $n\rightarrow\infty$.
\end{proof}
\begin{proposition}
\label{intI0}
For all $x\geqslant2$,
\begin{equation*}
\int_0^\infty\left(e^{-t}-e^{-xt}I_0(2t)\right)\frac{dt}{t}=\argcosh(x/2).
\end{equation*}
\end{proposition}
\begin{proof}
Setting $x=0$ in (\ref{seriesrepI}), we have
\begin{equation*}
I_0(2t)=\sum_{n\geqslant0}\frac{t^{2n}}{(n!)^2}.
\end{equation*}
It follows
\begin{align*}
\int_0^\infty e^{-xt}(I_0(2t)-1)\frac{dt}{t}&=\int_0^\infty e^{-xt}\sum_{n\geqslant1}\frac{t^{2n}}{(n!)^2}\frac{dt}{t}\\
&=\sum_{n\geqslant1}\frac{(2n-1)!}{(n!)^2}\frac{1}{x^{2n}}.
\end{align*}
Let $y=1/x^2$ with $y\leqslant1/4$, so the above is equivalent to the following sum $\sum_{n\geqslant1}y^n(2n-1)!/(n!)^2$.\\
Let $C_n=C_{2n}^n/(n+1)=(2n)!/(n+1)!n!$ be the Catalan numbers, $n\geqslant0$, where $C_m^n=m!/n!(m-n)!$ is the binomial coefficient. The generating function of the Catalan numbers is given by
\begin{equation}
\label{catalan}
\sum_{n\geqslant0}C_ny^n=\frac{2}{1+\sqrt{1-4y}}.
\end{equation}
The integration over $y$ of the above leads to
\begin{equation*}
\sum_{n\geqslant0}\frac{C_n}{n+1}y^{n+1}=\log(1+\sqrt{1-4y})-\sqrt{1-4y}+\textrm{constant}.
\end{equation*}
Taking the limit $y\rightarrow0$ on both sides gives the $\textrm{constant}=1-\log2$. Hence,
\begin{align*}
\sum_{n\geqslant0}\frac{C_n}{n+1}y^{n+1}&=y+\sum_{n\geqslant2}\frac{(2n-2)!}{(n!)^2}y^n\\
&=\log(1+\sqrt{1-4y})-\sqrt{1-4y}+1-\log2.
\end{align*}
Let $\alpha_n=C_{n-1}/n=(2n-2)!/(n!)^2$, $n\geqslant2$, and $\alpha_1=1$, and let $g(y)=\log(1+\sqrt{1-4y})-\sqrt{1-4y}+1-\log2$. So the previous equation can be written as
\begin{equation*}
\sum_{n\geqslant1}\alpha_ny^n=g(y).
\end{equation*}
So (\ref{catalan}) is equivalent to
\begin{equation*}
\sum_{n\geqslant1}n\alpha_ny^{n-1}=g'(y).
\end{equation*}
Finally,
\begin{align*}
\sum_{n\geqslant1}\frac{(2n-1)!}{(n!)^2}y^n&=\sum_{n\geqslant1}(2n-1)\alpha_ny^n\\
&=2y\sum_{n\geqslant1}n\alpha_ny^{n-1}-\sum_{n\geqslant1}\alpha_ny^n\\
&=2yg'(y)-g(y)\\
&=\log\left(\frac{2}{1+\sqrt{1-4y}}\right).
\end{align*}
Writing the above in terms of $x$ gives for all $x\geqslant2$,
\begin{equation*}
\int_0^\infty e^{-xt}(I_0(2t)-1)\frac{dt}{t}=\log\frac{x}{2}+\log(x-\sqrt{x^2-4}).
\end{equation*}
Notice that the above is the generating function of the Catalan numbers, and therefore is equal to $\log(\sum_{n\geqslant0}C_nx^{-2n})$.\\
Using the following integral identity for all $x\in\mathbb{C}$ with $\textrm{Re}(x)>0$
\begin{equation*}
\int_0^\infty\left(e^{-t}-e^{-xt}\right)\frac{dt}{t}=\log x
\end{equation*}
one has
\begin{equation*}
\int_0^\infty\left(e^{-t}-e^{-xt}I_0(2t)\right)\frac{dt}{t}=\log\left(\frac{x+\sqrt{x^2-4}}{2}\right)=\argcosh(x/2).
\end{equation*}
\end{proof}
\subsection{\texorpdfstring{\boldmath$d$-dimensional modified $I$-Bessel function\unboldmath\ }{d-dimensional modified I-Bessel function}}
\label{ddimB}
Let $m, p_1,\ldots,p_d$ be positive integers. By analogy with the two-dimensional $J$-Bessel function defined in \cite{korsch2006two} we define the $d$-dimensional modified $I$-Bessel function of order $m$, $I_m^{p_1,\ldots,p_d}(u_1,\ldots,\allowbreak u_d)$, as the generating function of $e^{\sum_{i=1}^du_i\cos{p_it}}$, that is
\begin{equation*}
e^{\sum_{i=1}^du_i\cos{p_it}}=\sum_{m=-\infty}^\infty I_m^{p_1,\ldots,p_d}(u_1,\ldots,u_d)e^{imt}.
\end{equation*}
In our computation we will only need $u_1=\ldots=u_d=2n^2t$ so we set $u_1=\ldots=u_d=u$. We have
\begin{equation*}
I_m^{p_1,\ldots,p_d}(u,\ldots,u)=\frac{1}{2\pi}\int_{-\pi}^\pi\sum_{(\mu_1,\ldots,\mu_d)\in\mathbb{Z}^d}\prod_{i=1}^dI_{\mu_i}(u)e^{i\left(\sum_{i=1}^d\mu_ip_i-m\right)t}dt.
\end{equation*}
The integral is non-zero only for $\displaystyle\sum_{i=1}^d\mu_ip_i=m$. Let $(\mu_1,\ldots,\mu_d)=(M_1,\ldots,M_d)$ be a particular solution, then the set of solutions is given by
\begin{equation*}
\mu_1=M_1-\sum_{i=2}^dp_ik_i,\quad\mu_i=M_i+p_1k_i,\quad i=2,\ldots,d,\quad k_2,\ldots,k_d\in\mathbb{Z}.
\end{equation*}
So we have
\begin{equation*}
I_m^{p_1,\ldots,p_d}(u,\ldots,u)=\sum_{(k_2,\ldots,k_d)\in\mathbb{Z}^{d-1}}I_{M_1-\sum_{i=2}^dp_ik_i}(u)\prod_{i=2}^dI_{M_i+p_1k_i}(u).
\end{equation*}
Let $\Gamma\defeq\{1,\gamma_1,\ldots,\gamma_{d-1}\}$ be a set of integral parameters, and $k_1\in\mathbb{N}$. We set $M_1=nk_1$, $M_2=\ldots=M_d=0$, $p_1=1$, $p_i=\gamma_{i-1}$, $i=2,\ldots,d$, then the $d$-dimensional modified $I$-Bessel function of order $nk_1$ and parameters set $\Gamma$ is given by
\begin{equation*}
I_{nk_1}^\Gamma(u,\ldots,u)\defeq I_{nk_1}^{1,\gamma_1,\ldots,\gamma_{d-1}}(u,\ldots,u)=\sum_{(k_2,\ldots,k_d)\in\mathbb{Z}^{d-1}}I_{nk_1-\sum_{i=1}^{d-1}\gamma_ik_{i+1}}(u)\prod_{i=2}^dI_{k_i}(u)
\end{equation*}
which has the integral representation
\begin{equation}
I_{nk_1}^\Gamma(u,\ldots,u)=\frac{1}{2\pi}\int_{-\pi}^\pi e^{u\left(\cos w+\sum_{i=1}^{d-1}\cos \gamma_iw\right)}e^{-ink_1w}dw.
\label{intrepI}
\end{equation}
Since $I_{-n}(u)=I_n(u)$, notice that
\begin{equation}
\label{symI}
I_{-nk_1}^\Gamma(u,\ldots,u)=I_{nk_1}^\Gamma(u,\ldots,u).
\end{equation}
\subsection{\texorpdfstring{Upper bounds for \boldmath$I$-Bessel functions\unboldmath\ }{Upperbounds on I-Bessel function}}
\label{upperbounds}
Recall Remark $4.2$ in \cite{chinta2010zeta}: For all $t>0$ the following upper bound holds:
\begin{equation}
\label{boundI0}
0\leqslant ne^{-n^2t}I_0(n^2t)\leqslant Ct^{-1/2}
\end{equation}
for some positive constant $C$.\\
Recall Lemma $4.6$ in \cite{chinta2010zeta}:
\begin{lemma}
\label{lemma4.6}
Fix $t\geqslant0$ and non-negative integers $x$ and $n_0$. Then for all $n\geqslant n_0$, we have the uniform bound
\begin{equation*}
0\leqslant\sqrt{n^2t}e^{-n^2t}I_{nx}(n^2t)\leqslant\left(\frac{n_0t}{x+n_0t}\right)^{n_0x/2}=\left(1+\frac{x}{n_0t}\right)^{-n_0x/2}\leqslant1.
\end{equation*}
\end{lemma}
\subsection{Method}
\label{method}
The method developed in \cite{chinta2010zeta} consists in studying the asymptotic behaviour of the Gauss transform of the theta function evaluated at zero in order to obtain the product of the Laplacian eigenvalues. This leads to the two following theorems which are adapted from Theorem $3.6$ in \cite{chinta2010zeta}. They express the logarithm of the determinant of the combinatorial Laplacian on the corresponding discrete torus in terms of integrals of theta and $I$-Bessel functions. The study of the asymptotics of these integrals will therefore lead to the asymptotic behaviour of the number of spanning trees.\\
In the case of the circulant graph we have:
\begin{theorem}
\label{methCirc}
We have the identity
\begin{equation*}
\log\Big(\prod_{\lambda_j\neq0}\lambda_j\Big)=n\mathcal{I}_d^\Gamma+\mathcal{H}_{C_n^\Gamma}
\end{equation*}
where
\begin{equation*}
\mathcal{I}_d^\Gamma=\int_0^\infty\left(e^{-t}-e^{-2dt}I_0^\Gamma(2t,\ldots,2t)\right)\frac{dt}{t}
\end{equation*}
and
\begin{equation*}
\mathcal{H}_{C_n^\Gamma}=-\int_0^\infty\left(\theta_{C_n^\Gamma}(t)-ne^{-2dt}I_0^\Gamma(2t,\ldots,2t)-1+e^{-t}\right)\frac{dt}{t}.
\end{equation*}
\end{theorem}
And in the case of the diagonal discrete torus we have:
\begin{theorem}
\label{methDis}
We have the identity
\begin{equation*}
\log\Big(\prod_{\lambda_j\neq0}\lambda_j\Big)=\det(\Lambda_n)\mathcal{I}_d^{\{\alpha_i\}_{i=1}^p}+\mathcal{H}_{\Lambda_n}
\end{equation*}
where
\begin{equation*}
\mathcal{I}_d^{\{\alpha_i\}_{i=1}^p}=\int_0^\infty\left(e^{-t}-e^{-2dt}I_0(2t)^{d-p}\sum_{(k_1,\ldots,k_p)\in\mathbb{Z}^p}\prod_{i=1}^pI_{k_i\alpha_ia_n}(2t)\right)\frac{dt}{t}
\end{equation*}
and
\begin{equation}
\label{Htermdt}
\mathcal{H}_{\Lambda_n}=-\int_0^\infty\left(\theta_{\Lambda_n}(t)-e^{-2dt}I_0(2t)^{d-p}\sum_{(k_1,\ldots,k_p)\in\mathbb{Z}^p}\prod_{i=1}^pI_{k_i\alpha_ia_n}(2t)-1+e^{-t}\right)\frac{dt}{t}.
\end{equation}
\end{theorem}
\section{Asymptotic behaviour of spectral determinant on circulant graphs}
\label{asc}
\subsection{Computation of the asymptotics}
Let $1\leqslant\gamma_1\leqslant\cdots\leqslant\gamma_{d-1}\leqslant\lfloor n/2\rfloor$ be positive integers and $C_n^\Gamma$ denote the circulant graph where $\Gamma\defeq\{1,\gamma_1,\ldots,\gamma_{d-1}\}$ is the set of generators. In this work we only consider circulant graphs with first generator equals to $1$. In this case one can verify that $C_n^\Gamma$ is isomorphic to the $d$-dimensional discrete torus $\mathbb{Z}^d/\Lambda_\Gamma\mathbb{Z}^d$ where $\Lambda_\Gamma$ is the following matrix
\begin{equation*}
\Lambda_\Gamma=\left(\begin{array}{c|ccc}
n&-\gamma_1&\cdots&-\gamma_{d-1}\\
\hline
&\multicolumn{3}{c}{\multirow{3}{*}{\Large{$I_{d-1}$}}}\\
0&&&\\
&&&
\end{array}\right)
\end{equation*}
where $I_{d-1}$ is the identity matrix of order $d-1$. Indeed, all the points on the lattice $\Lambda_\Gamma\mathbb{Z}^d$ are identified according to the numbers, where the nearest neighbours are connected to each other. Denote by $e_i$, $i=1,\ldots,d$, the canonical basis of $\mathbb{Z}^d$. Then $0\in\mathbb{Z}^d$ is connected to $e_i$ and $-e_i$ for $i=1,\ldots,d$. For $v\in\mathbb{Z}/n\mathbb{Z}$, all the points $ve_1+\Lambda_\Gamma\mathbb{Z}^d$ are identified to $v$. Hence $0$ is connected to $1$. Since $-e_1=(n-1)e_1-\Lambda_\Gamma e_1$, $0$ is connected to $n-1$. Using that $e_{i+1}=\gamma_ie_1+\Lambda_\Gamma e_{i+1}$, $i=1,\ldots,d-1$, $0$ is connected to $\gamma_i$ for all $i=1,\ldots,d-1$. Finally, $-e_{i+1}=-\gamma_ie_1-\Lambda_\Gamma e_{i+1}$, $i=1,\ldots,d-1$, so that $0$ is connected to $-\gamma_i$ mod $n$, for all $i=1,\ldots,d-1$. Similarly, all $v\in\mathbb{Z}/n\mathbb{Z}$ are connected to $v\pm\gamma_i$ mod $n$ for all $i=1,\ldots,d$. Therefore the quotient $\mathbb{Z}^d/\Lambda_\Gamma\mathbb{Z}^d$ with nearest neighbours connected to each other is isomorphic to the circulant graph $C^\Gamma_n$. Figure \ref{latticeC127} illustrates the lattice corresponding to the circulant graph $C^{1,2}_7$ represented in Figure \ref{circ}.
\begin{figure}
\centering
\includegraphics[width=8cm]{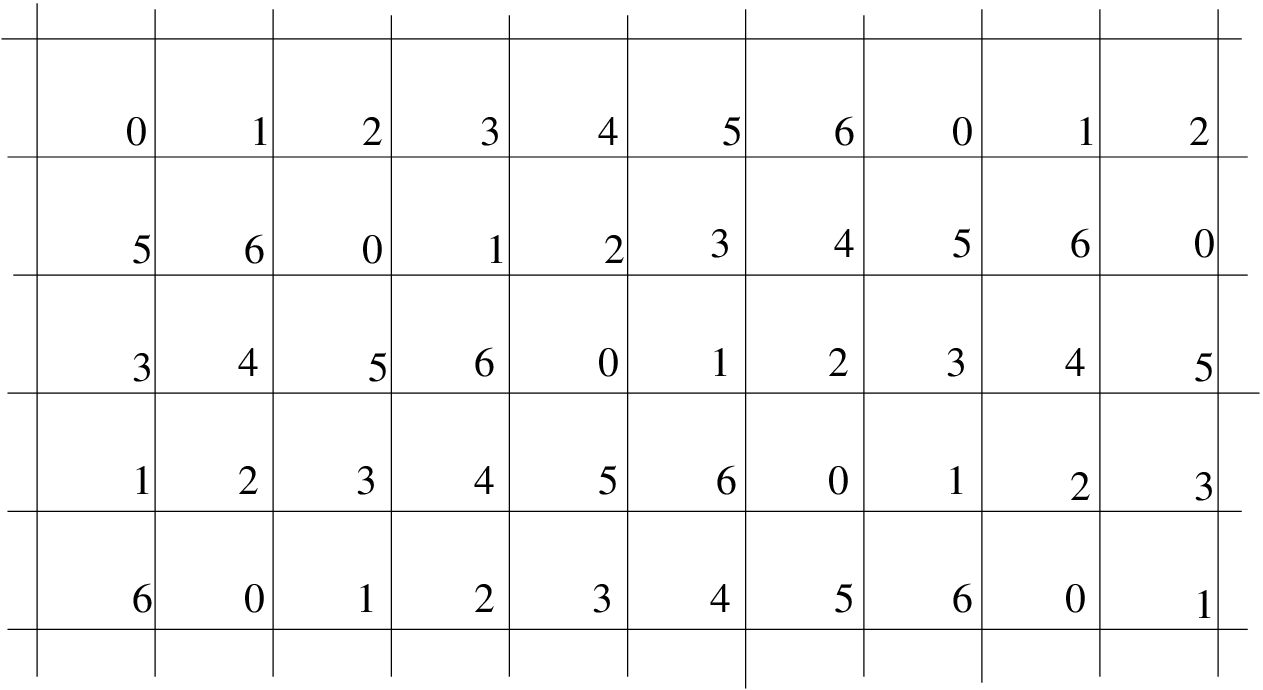}
\caption[lattice]{The lattice $\left(\begin{array}{cc}7&-2\\0&1\end{array}\right)\mathbb{Z}^2$.}
\label{latticeC127}
\end{figure}
The fact that the matrix is almost diagonal simplifies the expression of the theta function. Indeed from Proposition \ref{thetainv} the theta function on $C^\Gamma_n$ is given by
\begin{equation*}
\theta_{C^\Gamma_n}(n^2t)=ne^{-2dn^2t}\sum_{(k_1,\cdots,k_d)\in\mathbb{Z}^d}I_{nk_1-\sum_{i=1}^{d-1}\gamma_ik_{i+1}}(2n^2t)\prod_{i=2}^dI_{k_i}(2n^2t).
\end{equation*}
Rewriting it in terms of the $d$-dimensional modified $I$-Bessel function defined in section \ref{ddimB} we get
\begin{equation*}
\theta_{C_n^\Gamma}(n^2t)=ne^{-2dn^2t}\sum_{k_1\in\mathbb{Z}}I_{nk_1}^\Gamma(2n^2t,\ldots,2n^2t).
\end{equation*}
A circulant graph is the Cayley graph of a finite abelian group, so the eigenvectors of the Laplacian on $C_n^\Gamma$ are the characters
\begin{equation*}
\chi_j(x)=e^{2\pi ijx/n},\quad j=0,1,\ldots,n-1.
\end{equation*}
By applying the Laplacian on the characters, we obtain the eigenvalues
\begin{equation*}
\lambda_j=2d-2\cos(2\pi j/n)-2\sum_{i=1}^{d-1}\cos(2\pi\gamma_ij/n),\quad j=0,1,\ldots,n-1.
\end{equation*}
Therefore, by definition of the theta function (\ref{theta}) it can also be written as
\begin{align}
\label{theta2}
\theta_{C_n^\Gamma}(n^2t)&=\sum_{j=0}^{n-1}e^{-(2d-2\cos(2\pi j/n)-2\sum_{i=1}^{d-1}\cos(2\pi\gamma_ij/n))n^2t}\nonumber\\
&=\sum_{j=0}^{n-1}e^{-4(\sin^2(\pi j/n)+\sum_{i=1}^{d-1}\sin^2(\pi\gamma_ij/n))n^2t}.
\end{align}
\begin{proposition}
With the above notation we have for all $t\geqslant0$,
\label{thetalimit}
\begin{equation*}
\lim_{n\rightarrow\infty}\theta_{C_n^\Gamma}(n^2t)=\Theta_{S^1}(c_\Gamma t)
\end{equation*}
where $\Theta_{S^1}$ is the theta function on the circle $S^1=\mathbb{R}/\mathbb{Z}$ given by
\begin{equation*}
\Theta_{S^1}(t)=\frac{1}{\sqrt{4\pi t}}\sum_{k=-\infty}^\infty e^{-k^2/(4t)}.
\end{equation*}
\end{proposition}
\begin{proof}
From the theta inversion formula on $\mathbb{Z}/m\mathbb{Z}$ (Theorem $10$ in \cite{karlsson2006heat}) we have for any $z\in\mathbb{C}$, and integers $x$ and $m>0$,
\begin{equation}
\label{PoissonTranslated}
\sum_{k=-\infty}^\infty I_{x+km}(z)=\frac{1}{m}\sum_{j=0}^{m-1}e^{\cos(2\pi j/m)z+2\pi ijx/m}.
\end{equation}
Using the expression of the theta function in terms of $I$-Bessel functions, it follows that for all $n\geqslant1$ and $t>0$,
\begin{align*}
\lvert\theta_{C_n^\Gamma}(n^2t)\rvert&=\lvert ne^{-2dn^2t}\sum_{(k_2,\ldots,k_d)\in\mathbb{Z}^{d-1}}\frac{1}{n}\sum_{j=0}^{n-1}e^{2n^2t\cos(2\pi j/n)-2\pi ij\sum_{i=1}^{d-1}\gamma_ik_{i+1}/n}\prod_{i=2}^dI_{k_i}(2n^2t)\rvert\\
&\leqslant\prod_{i=2}^d\sum_{k_i\in\mathbb{Z}}e^{-2n^2t}I_{k_i}(2n^2t)\sum_{j=0}^{n-1}e^{-2n^2t\left(1-\cos(2\pi j/n)\right)}\\
&\leqslant\sum_{j=0}^{n-1}e^{-8\pi^2ctj^2}\leqslant\sum_{j=0}^{n-1}e^{-c'tj}\leqslant\frac{1}{1-e^{-c't}}
\end{align*}
where $c'>0$. In the second inequality we used the fact that for all $v\in[0,\pi]$, $(1-\cos v)/v^2\geqslant c$, with $c=1/2-\pi^2/24>0$, and $e^{-t}\sum_{x\in\mathbb{Z}}I_x(t)=1$.\\
It follows that
\begin{equation}
\label{limtheta}
\lim_{n\rightarrow\infty}\theta_{C_n^\Gamma}(n^2t)=\sum_{k_1\in\mathbb{Z}}\lim_{n\rightarrow\infty}ne^{-2dn^2t}I_{nk_1}^\Gamma(2n^2t,\ldots,2n^2t).
\end{equation}
Let $k_1>0$. From the integral representation of the $d$-dimensional $I$-Bessel function we have
\begin{equation*}
ne^{-2dn^2t}I_{nk_1}^\Gamma(2n^2t,\ldots,2n^2t)=\frac{1}{2\pi k_1}\int_{-\pi nk_1}^{\pi nk_1}e^{iw}e^{-2n^2t\left(d-\cos(w/(nk_1))-\sum_{i=1}^{d-1}\cos(\gamma_iw/(nk_1))\right)}dw.
\end{equation*}
Since $(1-\cos v)/v^2\geqslant c>0$ for all $v\in[0,\pi]$, we have that
\begin{equation*}
n^2(d-\cos(w/(nk_1))-\sum_{i=1}^{d-1}\cos(\gamma_iw/(nk_1)))\geqslant 
c\left(w/k_1\right)^2
\end{equation*}
for all $w\in[0,\pi nk_1]$. Hence for all $n\geqslant1$,
\begin{equation*}
\lvert ne^{-2dn^2t}I_{nk_1}^\Gamma(2n^2t,\ldots,2n^2t)\rvert\leqslant\frac{1}{2\pi k_1}\int_{-\pi nk_1}^{\pi nk_1}e^{-2tcw^2/k_1^2}dw\leqslant\frac{1}{2\pi k_1}\int_{-\infty}^\infty e^{-2tcw^2/k_1^2}dw=\sqrt{\frac{2}{\pi ct}}.
\end{equation*}
We also have that
\begin{equation*}
\lim_{n\rightarrow\infty}n^2(d-\cos(w/(nk_1))-\sum_{i=1}^{d-1}\cos(\gamma_iw/(nk_1)))=\frac{c_\Gamma}{2}(w/k_1)^2.
\end{equation*}
So by the Lebesgue dominated convergence Theorem, we have for all $k_1>0$
\begin{align}
\label{limthetak1}
\lim_{n\rightarrow\infty}ne^{-2dn^2t}I_{nk_1}^\Gamma(2n^2t,\ldots,2n^2t)&=\frac{1}{2\pi k_1}\int_{-\infty}^\infty e^{-c_\Gamma tw^2/k_1^2}e^{iw}dw\nonumber\\
&=\frac{1}{\sqrt{4\pi c_\Gamma t}}e^{-k_1^2/(4c_\Gamma t)}.
\end{align}
Let $k_1=0$. From the integral representation of the $d$-dimensional $I$-Bessel function we have
\begin{equation*}
ne^{-2dn^2t}I_0^\Gamma(2n^2t,\ldots,2n^2t)=\frac{1}{2\pi}\int_{-\pi n}^{\pi n}e^{-2n^2t\left(d-\cos(w/n)-\sum_{i=1}^{d-1}\cos(\gamma_iw/n)\right)}dw.
\end{equation*}
With the same argument as in the case $k_1>0$ we can apply the Lebesgue dominated convergence Theorem and we get
\begin{align}
\label{limtheta0}
\lim_{n\rightarrow\infty}ne^{-2dn^2t}I_0^\Gamma(2n^2t,\ldots,2n^2t)&=\frac{1}{2\pi}\int_{-\infty}^\infty e^{-c_\Gamma tw^2}dw\nonumber\\
&=\frac{1}{\sqrt{4\pi c_\Gamma t}}.
\end{align}
Putting (\ref{limthetak1}) and (\ref{limtheta0}) in (\ref{limtheta}) and using (\ref{symI}), the result follows.
\end{proof}
\begin{proposition}
With the above notation we have
\label{int01theta-I0}
\begin{align*}
&\lim_{n\rightarrow\infty}\int_0^1(\theta_{C_n^\Gamma}(n^2t)-ne^{-2dn^2t}I_0^\Gamma(2n^2t,\ldots,2n^2t))\frac{dt}{t}\\
&=\int_0^1\Big(\Theta_{S^1}(c_\Gamma t)-\frac{1}{\sqrt{4\pi c_\Gamma t}}\Big)\frac{dt}{t}.
\end{align*}
\end{proposition}
\begin{proof}
For a given positive integer $k_1\geqslant1$, let $D_{k_1}$ denote the following set
\begin{equation*}
D_{k_1}=\{(k_2,\ldots,k_d)\in\mathbb{Z}^{d-1}\mid|nk_1-\sum_{i=1}^{d-1}\gamma_ik_{i+1}|\leqslant nk_1/2\}
\end{equation*}
and let $D_{k_1}^c=\mathbb{Z}^{d-1}\setminus D_{k_1}$ denote the complement of $D_{k_1}$. From the theta inversion formula we have
\begin{align*}
&\theta_{C_n^\Gamma}(n^2t)-ne^{-2dn^2t}I_0^\Gamma(2n^2t,\ldots,2n^2t)\\
&=2ne^{-2dn^2t}\sum_{k_1=1}^\infty\sum_{(k_2,\ldots,k_d)\in\mathbb{Z}^{d-1}}I_{nk_1-\sum_{i=1}^{d-1}\gamma_ik_{i+1}}(2n^2t)\prod_{i=2}^{d}I_{k_i}(2n^2t)\\
&=2ne^{-2dn^2t}\sum_{k_1=1}^\infty\Bigg[\sum_{(k_2,\ldots,k_d)\in D_{k_1}^c}+\sum_{(k_2,\ldots,k_d)\in D_{k_1}}\Bigg]I_{nk_1-\sum_{i=1}^{d-1}\gamma_ik_{i+1}}(2n^2t)\prod_{i=2}^{d}I_{k_i}(2n^2t).
\end{align*}
Since the modified Bessel function $I_k$ is decreasing in the index $k$ \cite{MR0213624}, for $(k_2,\ldots,k_d)\in D_{k_1}^c$,
\begin{equation*}
I_{nk_1-\sum_{i=1}^{d-1}\gamma_ik_{i+1}}(2n^2t)=I_{\lvert nk_1-\sum_{i=1}^{d-1}\gamma_ik_{i+1}\rvert}(2n^2t)\leqslant I_{nk_1/2}(2n^2t).
\end{equation*}
Using that $e^{-2n^2t}\sum_{k\in\mathbb{Z}}I_k(2n^2t)=1$, it follows that
\begin{equation*}
2ne^{-2dn^2t}\sum_{k_1=1}^\infty\sum_{(k_2,\ldots,k_d)\in D_{k_1}^c}I_{nk_1-\sum_{i=1}^{d-1}\gamma_ik_{i+1}}(2n^2t)\prod_{i=2}^{d}I_{k_i}(2n^2t)\leqslant2ne^{-2dn^2t}\sum_{k_1=1}^\infty I_{nk_1/2}(2n^2t).
\end{equation*}
Using Lemma \ref{lemma4.6}, for all $n\geqslant n_0$ the above is less or equal than
\begin{align}
\sqrt{\frac{2}{t}}\sum_{k=1}^{\infty}\left(1+\frac{k}{4n_0t}\right)^{-n_0k/4}\leqslant\sqrt{\frac{2}{t}}\frac{1}{(1+1/(4n_0t))^{n_0/4}-1}\leqslant\sqrt{2}(4n_0)^{n_0/4}t^{n_0/4-1/2}.
\label{bound1}
\end{align}
For $(k_2,\ldots,k_d)\in D_{k_1}$, we have
\begin{equation*}
|nk_1-\sum_{i=1}^{d-1}\gamma_i\lvert k_{i+1}\rvert|\leqslant\frac{nk_1}{2}.
\end{equation*}
Since $1\leqslant\gamma_1\leqslant\cdots\leqslant\gamma_{d-1}$, it follows that
\begin{equation*}
\frac{nk_1}{2}\leqslant\sum_{i=1}^{d-1}\gamma_i\lvert k_{i+1}\rvert\leqslant\gamma_{d-1}(d-1)\max_{i\in\{2,\ldots,d\}}\lvert k_i\rvert
\end{equation*}
so that
\begin{equation}
\label{max}
\max_{i\in\{2,\ldots,d\}}\lvert k_i\rvert\geqslant\frac{nk_1}{2(d-1)\gamma_{d-1}}.
\end{equation}
Let $S_{d-1}$ denote the set of permutations of $\{2,\ldots,d\}$. By ordering the $k_i$'s in the second summation we obtain
\begin{align*}
&\lvert2ne^{-2dn^2t}\sum_{k_1=1}^\infty\sum_{(k_2,\ldots,k_d)\in D_{k_1}}I_{nk_1-\sum_{i=1}^{d-1}\gamma_ik_{i+1}}(2n^2t)\prod_{i=2}^{d}I_{k_i}(2n^2t)\rvert\\
&=\lvert2ne^{-2dn^2t}\sum_{k_1=1}^\infty\sum_{\sigma\in S_{d-1}}\Bigg[\sum_{\substack{(k_2,\ldots,k_d)\in D_{k_1}\\ \lvert k_{\sigma(2)}\rvert<\cdots<\lvert k_{\sigma(d)}\rvert}}\\
&\quad\quad\quad\quad\quad\quad\quad\quad\quad\quad\quad\quad\quad\quad+\sum_{\substack{(k_2,\ldots,k_d)\in D_{k_1}\\ \lvert k_{\sigma(2)}\rvert=\cdots=\lvert k_{\sigma(d)}\rvert}}\Bigg]I_{nk_1-\sum_{i=2}^d\gamma_{\sigma(i)-1}k_{\sigma(i)}}(2n^2t)\prod_{i=2}^{d}I_{k_\sigma(i)}(2n^2t)\rvert
\end{align*}
Using inequatity (\ref{max}) and the fact that $I_k$ is decreasing in $k$, we have
\begin{equation*}
I_{k_{\sigma(d)}}(2n^2t)=I_{\lvert k_{\sigma(d)}\rvert}(2n^2t)\leqslant I_{nk_1/(2(d-1)\gamma_{d-1})}(2n^2t)
\end{equation*}
hence the above is less or equal than
\begin{align}
2n&e^{-2n^2t}\sum_{k_1=1}^\infty I_{nk_1/(2(d-1)\gamma_{d-1})}(2n^2t)\nonumber\\
&\times\sum_{\sigma\in S_{d-1}}\sum_{(k_{\sigma(2)},\ldots,k_{\sigma(d-1)})\in\mathbb{Z}^{d-2}}\lvert e^{-2n^2t}\sum_{k_{\sigma(d)}\in\mathbb{Z}}I_{nk_1-\sum_{i=2}^d\gamma_{\sigma(i)-1}k_{\sigma(i)}}(2n^2t)\rvert\prod_{i=2}^{d-1}e^{-2n^2t}I_{k_{\sigma(i)}}(2n^2t).
\label{2ndsum}
\end{align}
Using the theta inversion formula (\ref{PoissonTranslated}) with $m=\gamma_{\sigma(d)-1}$ and $x=nk_1-\sum_{i=2}^{d-1}\gamma_{\sigma(i)-1}k_{\sigma(i)}$, we have that
\begin{align*}
&\lvert e^{-2n^2t}\sum_{k_{\sigma(d)}\in\mathbb{Z}}I_{nk_1-\sum_{i=2}^d\gamma_{\sigma(i)-1}k_{\sigma(i)}}(2n^2t)\rvert\\
&=\frac{1}{\gamma_{\sigma(d)-1}}\lvert\sum_{k=0}^{\gamma_{\sigma(d)-1}-1}e^{-2n^2t(1-\cos(2\pi k/\gamma_{\sigma(d)-1}))+2\pi ik(nk_1-\sum_{i=2}^{d-1}\gamma_{\sigma(i)-1}k_{\sigma(i)})/\gamma_{\sigma(d)-1}}\rvert\leqslant1.
\end{align*}
Putting the above in (\ref{2ndsum}) and using that $e^{-2n^2t}\sum_{k\in\mathbb{Z}}I_k(2n^2t)=1$, the second summation is less or equal than
\begin{align}
&2ne^{-2n^2t}\sum_{k_1=1}^\infty I_{nk_1/(2(d-1)\gamma_{d-1})}(2n^2t)(d-1)!\nonumber\\
&\leqslant\sqrt{2}(d-1)!(4(d-1)\gamma_{d-1}n_0)^{n_0/(4(d-1)\gamma_{d-1})}t^{n_0/(4(d-1)\gamma_{d-1})-1/2}
\label{bound2}
\end{align}
for all $n\geqslant n_0$, where we used Lemma \ref{lemma4.6} in the second inequality. Inequalities (\ref{bound1}) and (\ref{bound2}) together lead to
\begin{align*}
&\lvert\theta_{C_n^\Gamma}(n^2t)-ne^{-2dn^2t}I_0^\Gamma(2n^2t,\ldots,2n^2t)\rvert\\
&\leqslant\sqrt{2}(4n_0)^{n_0/4}t^{n_0/4-1/2}+\sqrt{2}(d-1)!(4(d-1)\gamma_{d-1}n_0)^{n_0/(4(d-1)\gamma_{d-1})}t^{n_0/(4(d-1)\gamma_{d-1})-1/2}
\end{align*}
which is integrable on $(0,1)$ with respect to the measure $dt/t$ for all $n\geqslant n_0=2(d-1)\gamma_{d-1}+1$. The proposition then follows from the Lebesgue dominated convergence Theorem and from the pointwise convergence.
\end{proof}
\noindent
Recall the following lemma from \cite{chinta2010zeta}:
\begin{lemma}
\label{int01exp}
For $n\in\mathbb{R}$, we have the asymptotic formula
\begin{equation*}
\int_0^1(e^{-n^2t}-1)\frac{dt}{t}=\Gamma'(1)-2\log n+o(1)\quad\textrm{as }n\rightarrow\infty.
\end{equation*}
\end{lemma}
\begin{proposition}
With the above notation we have that
\label{liminttheta-1}
\begin{equation*}
\lim_{n\rightarrow\infty}\int_1^\infty\left(\theta_{C_n^\Gamma}(n^2t)-1\right)\frac{dt}{t}=\int_1^\infty\big(\Theta_{S^1}(c_\Gamma t)-1\big)\frac{dt}{t}.
\end{equation*}
\end{proposition}
\begin{proof}
From Proposition \ref{thetalimit} we have for all $t>0$, the pointwise limit
\begin{equation*}
\lim_{n\rightarrow\infty}\theta_{C_n^\Gamma}(n^2t)-1=\Theta_{S^1}(c_\Gamma t)-1.
\end{equation*}
From (\ref{theta2}) we have
\begin{equation*}
\theta_{C_n^\Gamma}(n^2t)=1+\sum_{j=1}^{n-1}e^{-4\sin^2(\pi j/n)n^2t}\prod_{i=1}^{d-1}e^{-4\sin^2(\pi\gamma_ij/n)n^2t}.
\end{equation*}
Since the product on $i$ is smaller than $1$, we have
\begin{equation*}
\theta_{C_n^\Gamma}(n^2t)\leqslant1+\sum_{j=1}^{n-1}e^{-4\sin^2(\pi j/n)n^2t}=1+2\sum_{j=1}^{\lfloor n/2\rfloor}e^{-4\sin^2(\pi j/n)n^2t}.
\end{equation*}
Using the elementary bound
\begin{equation*}
\sin(\pi x)\geqslant\pi x\left(1-\pi^2x^2/6\right)\geqslant c\pi x
\end{equation*}
for all $x\in[0,1/2]$, where $c=1-\pi^2/24>0$, we have
\begin{equation*}
\theta_{C_n^\Gamma}(n^2t)-1\leqslant2\sum_{j=1}^{\lfloor n/2\rfloor}e^{-4c^2\pi^2j^2t}\leqslant2\sum_{j=1}^{\infty}e^{-djt}=\frac{2}{e^{dt}-1}\leqslant\frac{2}{1-e^{-d}}e^{-dt},
\end{equation*}
for all $t\geqslant1$, where $d=4c^2\pi^2>0$. Since it is integrable on $(1,\infty)$ with respect to the measure $dt/t$, the proposition follows from the Lebesgue dominated convergence Theorem.
\end{proof}
\begin{proposition}
With the above notation we have
\label{limintI0}
\begin{equation*}
\lim_{n\rightarrow\infty}\int_1^\infty ne^{-2dn^2t}I_0^\Gamma(2n^2t,\ldots,2n^2t)\frac{dt}{t}=\frac{1}{\sqrt{\pi c_\Gamma}}.
\end{equation*}
\end{proposition}
\begin{proof}
By definition, we have
\begin{equation*}
I_0^\Gamma(2n^2t,\ldots,2n^2t)=\sum_{(k_2,\ldots,k_d)\in\mathbb{Z}^{d-1}}I_{-\sum_{i=1}^{d-1}\gamma_ik_{i+1}}(2n^2t)\prod_{i=2}^dI_{k_i}(2n^2t).
\end{equation*}
From Lemma \ref{lemma4.6} we have the uniform upper bound
\begin{equation*}
ne^{-2n^2t}I_{-\sum_{i=1}^{d-1}\gamma_ik_{i+1}}(2n^2t)\leqslant\frac{1}{\sqrt{2t}}.
\end{equation*}
Hence
\begin{equation*}
ne^{-2dn^2t}I_0^\Gamma(2n^2t,\ldots,2n^2t)\leqslant\frac{1}{\sqrt{2t}}(e^{-2n^2t}\sum_{k\in\mathbb{Z}}I_k(2n^2t))^{d-1}=\frac{1}{\sqrt{2t}}
\end{equation*}
which is integrable on $(1,\infty)$ with respect to the measure $dt/t$. By the Lebesgue dominated convergence Theorem it follows that
\begin{equation*}
\lim_{n\rightarrow\infty}\int_1^\infty ne^{-2dn^2t}I_0^\Gamma(2n^2t,\ldots,2n^2t)\frac{dt}{t}=\int_1^\infty\frac{1}{\sqrt{4\pi c_\Gamma t}}\frac{dt}{t}=\frac{1}{\sqrt{\pi c_\Gamma}}.
\end{equation*}
\end{proof}
Since $\int_1^\infty e^{-n^2t}dt/t$ converges to zero as $n\rightarrow\infty$, putting Lemma \ref{int01exp} and Propositions \ref{int01theta-I0}, \ref{liminttheta-1} and \ref{limintI0} together in Theorem \ref{methCirc} leads to the asymptotic of the $\mathcal{H}_{C_n^\Gamma}$ term as $n\rightarrow\infty$:
\begin{align*}
\mathcal{H}_{C_n^\Gamma}&=2\log n-\int_0^{c_\Gamma}(\Theta_{S^1}(t)-\frac{1}{\sqrt{4\pi t}})\frac{dt}{t}-\Gamma'(1)-\int_{c_\Gamma}^\infty(\Theta_{S^1}(t)-1)\frac{dt}{t}+\frac{1}{\sqrt{\pi c_\Gamma}}+o(1).
\end{align*}
Using equation (\ref{zetaS1}) we can then rewrite:
\begin{equation*}
\mathcal{H}_{C_n^\Gamma}=2\log n-\zeta'_{S^1}(0)-\log c_\Gamma+o(1)\quad\textrm{as }n\rightarrow\infty.
\end{equation*}
Since $\zeta'_{S^1}(0)=0$ (\ref{zeta'_S1(0)}) we get
\begin{equation*}
\mathcal{H}_{C_n^\Gamma}=2\log n-\log c_\Gamma+o(1)\quad\textrm{as }n\rightarrow\infty
\end{equation*}
and so
\begin{equation*}
\log\textrm{det}^\ast\Delta_{C_n^\Gamma}=n\int_0^\infty(e^{-t}-e^{-2dt}I_0^\Gamma(2t,\ldots,2t))\frac{dt}{t}+2\log n-\log c_\Gamma+o(1)\quad\textrm{as }n\rightarrow\infty
\end{equation*}
which proves Theorem \ref{ThCirc}.
\subsection{Asymptotic number of spanning trees and comparison of the results}
Notice that in the trivial case $d=1$, the cycle has $n$ spanning trees so $\log\textrm{det}^\ast\Delta_{C_n}=\log n^2$. On the other hand, from Proposition \ref{intI0}
\begin{equation*}
\int_0^\infty(e^{-t}-e^{-2t}I_0(2t))\frac{dt}{t}=0
\end{equation*}
and so the right hand side of the asymptotic development is $2\log n$. Therefore the theorem is verified in this particular case.\\
From Kirchhoff's matrix tree theorem and Theorem \ref{ThCirc}, the number of spanning trees in the circulant graph $C_n^\Gamma$ with $\Gamma=\{1,\gamma_1,\ldots,\gamma_{d-1}\}$ is asymptotically given by
\begin{equation}
\tau(C_n^\Gamma)=\frac{n}{c_\Gamma}e^{n\mathcal{I}_d^\Gamma+o(1)}\textrm{ as }n\rightarrow\infty.
\label{taucirc}
\end{equation}
The lead term can be rewritten as
\begin{equation*}
\mathcal{I}_d^\Gamma=\int_0^\infty(e^{-t}-e^{-2dt}I_0^\Gamma(2t,\ldots,2t))\frac{dt}{t}=\log(2d)+\int_0^\infty e^{-2dt}(1-I_0^\Gamma(2t,\ldots,2t))\frac{dt}{t}.
\end{equation*}
From the integral representation of $I_0^\Gamma$ (\ref{intrepI}) and writing the exponential as a series one has
\begin{align*}
\int_0^\infty e^{-2dt}(1-I_0^\Gamma(2t,\ldots,2t))\frac{dt}{t}&=-\frac{1}{2\pi}\int_0^\infty e^{-2dt}\sum_{n=1}^\infty\frac{2^n}{n!}\int_{-\pi}^\pi(\cos w+\sum_{i=1}^{d-1}\cos(\gamma_iw))^ndwt^n\frac{dt}{t}\\
&=-\frac{1}{2\pi}\sum_{n=1}^\infty\frac{1}{d^n}\frac{1}{n}\int_{-\pi}^\pi(\cos w+\sum_{i=1}^{d-1}\cos(\gamma_iw))^ndw\\
&=\frac{1}{2\pi}\int_{-\pi}^\pi\log\left(1-\frac{\cos w+\sum_{i=1}^{d-1}\cos(\gamma_iw)}{d}\right)dw\\
&=\int_0^1\log(\sin^2(\pi w)+\sum_{i=1}^{d-1}\sin^2(\pi\gamma_iw))dw+\log\frac{2}{d}.
\end{align*}
Hence the lead term is given by
\begin{equation*}
\mathcal{I}_d^\Gamma=\log4+\int_0^1\log(\sin^2(\pi w)+\sum_{i=1}^{d-1}\sin^2(\pi\gamma_iw))dw
\end{equation*}
which corresponds to Lemma $2$ of \cite{golin2010asymptotic}.\\
As mentioned in the introduction, the authors show in \cite{zhang1999number} that the number of spanning trees in a circulant graph is given by
\begin{equation*}
\tau(C_n^{\gamma_1,\ldots,\gamma_d})=na_n^2
\end{equation*}
where $a_n$ satisfies a recurrence relation which behaves asymptotically as $c\phi^n$ for some constants $c$ and $\phi$ which can be determined numerically. Comparing with (\ref{taucirc}) it follows that
\begin{equation*}
c^2=\frac{1}{c_\Gamma}
\end{equation*}
which is numerically verified with the values in Table 1 in \cite{zhang1999number}. This answers to one of the questions asked in the conclusion of \cite{atajan2006further}.
\section{Asymptotic behaviour of spectral determinant on degenerating tori}
\label{asdt}
We consider the sequence of $d$-dimensional discrete tori described in the introduction. For simplicity, we denote by $\theta_{\Lambda_n}$ the theta function associated to $\mathbb{Z}^d/\Lambda_n\mathbb{Z}^d$. It is given by
\begin{equation*}
\theta_{\Lambda_n}(t)=\sum_{\lambda_j}e^{-\lambda_jt}
\end{equation*}
where
\begin{align*}
\{\lambda_j\}_{j=0,1,\ldots,\det(\Lambda_n)-1}=\{2d&-2\sum_{i=1}^p\cos(2\pi m_i/(\alpha_ia_n))-2\sum_{i=1}^{d-p}\cos(2\pi m'_i/(\beta_in)):\\
&0\leqslant m_i<\alpha_ia_n,i=1,\ldots,p\textrm{ and }0\leqslant m'_i<\beta_in,i=1,\ldots,d-p\}
\end{align*}
are the eigenvalues of the combinatorial Laplacian on $\mathbb{Z}^d/\Lambda_n\mathbb{Z}^d$. From the theta inversion formula on $\mathbb{Z}^d/\Lambda_n\mathbb{Z}^d$ (Proposition \ref{thetainv}) we have for all $t\geqslant0$
\begin{equation}
\label{thetaDegT}
\theta_{\Lambda_n}(t)=\left(\prod_{i=1}^p\alpha_ia_ne^{-2t}\sum_{k\in\mathbb{Z}}I_{k\alpha_ia_n}(2t)\right)\left(\prod_{i=1}^{d-p}\beta_ine^{-2t}\sum_{k\in\mathbb{Z}}I_{k\beta_in}(2t)\right).
\end{equation}
\subsection{\texorpdfstring{Computation of the lead term when \boldmath$a_n$ grows sublinearly with respect to \boldmath$n$}{sublinear}}
Let $c_d$ be the integral below. A numerical estimation of it is discussed in section $7.2$ of \cite{chinta2010zeta}.
\begin{equation*}
c_d=\int_0^\infty\left(e^{-t}-e^{-2dt}I_0(2t)^d\right)\frac{dt}{t}.
\end{equation*}
The lead term of $\log\textnormal{det}^\ast\Delta_{\mathbb{Z}^d/\Lambda_n\mathbb{Z}^d}$ in Theorem \ref{methDis} is given by
\begin{align*}
&\det(\Lambda_n)\mathcal{I}_d^{\{\alpha_i\}_{i=1}^p}=\det(\Lambda_n)\int_0^\infty\Big(e^{-t}-e^{-2dt}I_0(2t)^{d-p}\sum_{(k_1,\ldots,k_p)\in\mathbb{Z}^p}\prod_{i=1}^pI_{k_i\alpha_ia_n}(2t)\Big)\frac{dt}{t}\\
&=\det(\Lambda_n)c_d-\det(\Lambda_n)\int_0^\infty e^{-2dt}I_0(2t)^{d-p}\sum_{(k_1,\ldots,k_p)\in\mathbb{Z}^p\backslash\{0\}}\left(\prod_{i=1}^pI_{k_i\alpha_ia_n}(2t)\right)\frac{dt}{t}\\
&=n^{d-p}a_n^p\det(\Lambda)c_d-\left(\frac{n}{a_n}\right)^{d-p}\det(B)\int_0^\infty J(a_n,t)dt
\end{align*}
where in the last equality the integration variable $t$ is changed into $a_n^2t$ and $J(a_n,t)$ is given by
\begin{equation*}
J(a_n,t)\defeq\frac{1}{t}\left(a_ne^{-2a_n^2t}I_0(2a_n^2t)\right)^{d-p}\sum_{(k_1,\ldots,k_p)\in\mathbb{Z}^p\backslash\{0\}}\prod_{i=1}^p\left(\alpha_ia_ne^{-2a_n^2t}I_{k_i\alpha_ia_n}(2a_n^2t)\right).
\end{equation*}
From Proposition \ref{prop4.7} we have that
\begin{equation*}
\lim_{n\rightarrow\infty}a_ne^{-2a_n^2t}I_0(2a_n^2t)=\frac{1}{\sqrt{4\pi t}}
\end{equation*}
and
\begin{equation*}
\lim_{n\rightarrow\infty}\alpha_ia_ne^{-2a_n^2t}I_{k_i\alpha_ia_n}(2a_n^2t)=\frac{\alpha_i}{\sqrt{4\pi t}}e^{-\alpha_i^2k_i^2/(4t)}.
\end{equation*}
From the definition of the zeta function (\ref{zeta2}), we have that
\begin{align*}
\int_0^\infty\frac{1}{(4\pi t)^{d/2}}\sum_{(k_1,\ldots,k_p)\in\mathbb{Z}^p\backslash\{0\}}e^{-\sum_{i=1}^p\alpha_i^2k_i^2/(4t)}\frac{dt}{t}&=\frac{1}{\pi^{d/2}}\Gamma(d/2)\sum_{(k_1,\ldots,k_p)\in\mathbb{Z}^p\backslash\{0\}}\frac{1}{(\sum_{i=1}^p\alpha_i^2k_i^2)^{d/2}}\\
&=(4\pi)^{d/2}\Gamma(d/2)\zeta_{\mathbb{R}^p/A^{-1}\mathbb{Z}^p}(d/2).
\end{align*}
So as $n\rightarrow\infty$,
\begin{equation*}
\int_0^\infty J(a_n,t)dt=(4\pi)^{d/2}\Gamma(d/2)\zeta_{\mathbb{R}^p/A^{-1}\mathbb{Z}^p}(d/2)+o(1).
\end{equation*}
The exchange of the limit as $n$ goes to infinity with the integral over $t$ can be justified using the same argument as in the proof of Proposition \ref{int01Theta} below on $(0,1)$ and Lemma \ref{lemma5.3} with inequality (\ref{boundI0}) on $(1,\infty)$. Hence as $n\rightarrow\infty$ the lead term behaves as
\begin{align*}
\det(\Lambda_n)\mathcal{I}_d^{\{\alpha_i\}_{i=1}^p}&=n^{d-p}a_n^p\textrm{det}(\Lambda)c_d-\left(\frac{n}{a_n}\right)^{d-p}\left(\det(\Lambda)(4\pi)^{d/2}\Gamma(d/2)\zeta_{\mathbb{R}^p/A^{-1}\mathbb{Z}^p}(d/2)+o(1)\right).
\end{align*}
\begin{remark}
To find one more term in the asymptotic development we would need to show that one can exchange the limit as $n\rightarrow\infty$ with the integration over $t$ of
\begin{equation*}
a_n^2(J(a_n,t)-(4\pi)^{d/2}\Gamma(d/2)\zeta_{\mathbb{R}^p/A^{-1}\mathbb{Z}^p}(d/2)).
\end{equation*}
We were not able to do that. One way of proving this is to find an upper bound of the above for all $n$ which is integrable over $t$ on $(0,\infty)$ and apply the Lebesgue dominated convergence Theorem. This means that we need a sharp integrable upper bound of $e^{-t}I_\nu(t)$. To the best of our knowledge, the best upper bound of $e^{-t}I_\nu(t)$ is given in \cite{balachandran2013exponential} and is not sharp enough. Assuming that one can exchange the limit with integration, the asymptotic development would be as $n\rightarrow\infty$:
\begin{align*}
&\log\textnormal{det}^\ast\Delta_{\mathbb{Z}^d/\Lambda_n\mathbb{Z}^d}=n^{d-p}a_n^p\textrm{det}(\Lambda)c_d\\
&-\left(\frac{n}{a_n}\right)^{d-p}\det(\Lambda)(4\pi)^{d/2}\Bigg[\Gamma(d/2)\zeta_{\mathbb{R}^p/A^{-1}\mathbb{Z}^p}(d/2)\\
&\qquad\qquad\qquad\qquad\qquad\quad+\frac{1}{a_n^2}\Bigg(-(d+4)\Gamma(d/2+1)\zeta_{\mathbb{R}^p/A^{-1}\mathbb{Z}^p}(d/2+1)\\
&\qquad\qquad\qquad\qquad\qquad\quad+\frac{4}{3}\Gamma(d/2+3)\sum_{i=1}^px_i^2\frac{\partial^2}{\partial x_i^2}\zeta_{\mathbb{R}^p/A(x)^{-1/2}\mathbb{Z}^p}(d/2+1)\Biggr\rvert_{\substack{x_i=\alpha_i^2\\i=1,\ldots,p}}\Bigg)+o\left(\frac{1}{a_n^2}\right)\Bigg]
\end{align*}
where $A(x)$ is the diagonal matrix $A(x)\defeq\textnormal{diag}(x_1,\ldots,x_p)$.
\end{remark}
\subsection{\texorpdfstring{Computation of the lead term when \boldmath$a_n$ is constant}{constant}}
\label{ancst}
When $a_n=1$, the lead term is given by
\begin{equation*}
\det(\Lambda_n)\mathcal{I}_d^{\{\alpha_i\}_{i=1}^p}=\det(\Lambda_n)\int_0^\infty\Big(e^{-t}-e^{-2dt}I_0(2t)^{d-p}\sum_{(k_1,\ldots,k_p)\in\mathbb{Z}^p}\prod_{i=1}^pI_{k_i\alpha_i}(2t)\Big)\frac{dt}{t}.
\end{equation*}
From the theta inversion formula (\ref{thetainvZ/mZ}) we have for $i=1,\ldots,p$
\begin{equation*}
\alpha_ie^{-2t}\sum_{k_i\in\mathbb{Z}}I_{k_i\alpha_i}(2t)=\sum_{j_i=0}^{\alpha_i-1}e^{-2(1-\cos(2\pi j_i/\alpha_i))t}.
\end{equation*}
Hence
\begin{equation*}
\det(\Lambda_n)\mathcal{I}_d^{\{\alpha_i\}_{i=1}^p}=n^{d-p}\det(B)\sum_{j=0}^{\det(A)-1}\int_0^\infty\left(e^{-t}-I_0(2t)^{d-p}e^{-(2(d-p)+\lambda_j)t}\right)\frac{dt}{t}
\end{equation*}
where
\begin{equation*}
\{\lambda_j\}_j=\{2p-2\sum_{i=1}^p\cos(2\pi j_i/\alpha_i):j_i=0,1,\ldots,\alpha_i-1,\textrm{ for }i=1,\ldots,p\},
\end{equation*}
$j=0,1,\ldots,\det(A)-1$, are the eigenvalues of the Laplacian on $\mathbb{Z}^p/A\mathbb{Z}^p$. 
\subsection{Asymptotic behaviour of the second term}
In this section we compute the asymptotics of the $\mathcal{H}_{\Lambda_n}$ term when $a_n$ indifferently goes to infinity sublinearly with respect to $n$ or is constant. To do this we change the integration variable $t$ into $n^2t$ in (\ref{Htermdt})
\begin{equation*}
\mathcal{H}_{\Lambda_n}=-\int_0^\infty\!\Big(\theta_{\Lambda_n}(n^2t)-\det(\Lambda_n)e^{-2dn^2t}I_0(2n^2t)^{d-p}\!\!\!\sum_{\substack{(k_1,\ldots,k_p)\\ \in\mathbb{Z}^p}}\!\prod_{i=1}^pI_{k_i\alpha_ia_n}(2n^2t)-1+e^{-n^2t}\Big)\frac{dt}{t}.
\end{equation*}
\begin{proposition}
\label{convTheta}
With the above notation, we have for all $t\geqslant0$,
\begin{equation*}
\lim_{n\rightarrow\infty}\theta_{\Lambda_n}(n^2t)=\Theta_{\mathbb{R}^{d-p}/B\mathbb{Z}^{d-p}}(t).
\end{equation*}
\end{proposition}
\begin{proof}
The theta function (\ref{thetaDegT}) with the change of variable is given by
\begin{equation*}
\theta_{\Lambda_n}(n^2t)=\left(\prod_{i=1}^p\alpha_ia_ne^{-2n^2t}\sum_{k\in\mathbb{Z}}I_{k\alpha_ia_n}(2n^2t)\right)\left(\prod_{i=1}^{d-p}\beta_ine^{-2n^2t}\sum_{k\in\mathbb{Z}}I_{k\beta_in}(2n^2t)\right).
\end{equation*}
From Proposition \ref{prop4.7bis} we have that
\begin{equation*}
\lim_{n\rightarrow\infty}\prod_{i=1}^p\alpha_ia_ne^{-2n^2t}\sum_{k\in\mathbb{Z}}I_{k\alpha_ia_n}(2n^2t)=1
\end{equation*}
and from Proposition \ref{prop4.7} we have
\begin{equation*}
\lim_{n\rightarrow\infty}\prod_{i=1}^{d-p}\beta_ine^{-2n^2t}I_{k\beta_in}(2n^2t)=\prod_{i=1}^{d-p}\frac{\beta_i}{\sqrt{4\pi t}}e^{-(\beta_ik)^2/(4t)}.
\end{equation*}
The proposition follows if we can exchange the limit with the sum. This can be justified in the same way as the proof of Proposition $5.2$ in \cite{chinta2010zeta}.
\end{proof}
\begin{proposition}
\label{int01Theta}
With the above notation, we have that
\begin{align*}
&\lim_{n\rightarrow\infty}\int_0^1\Big(\theta_{\Lambda_n}(n^2t)-\det(\Lambda_n)e^{-2dn^2t}I_0(2n^2t)^{d-p}\sum_{(k_1,\ldots,k_p)\in\mathbb{Z}^p}\prod_{i=1}^pI_{k_i\alpha_ia_n}(2n^2t)\Big)\frac{dt}{t}\\
&=\int_0^1\left(\Theta_{\mathbb{R}^{d-p}/B\mathbb{Z}^{d-p}}(t)-\frac{\det(B)}{(4\pi t)^{(d-p)/2}}\right)\frac{dt}{t}.
\end{align*}
\end{proposition}
\begin{proof}
From Propositions \ref{convTheta}, \ref{prop4.7} and \ref{prop4.7bis} we have the pointwise convergence:
\begin{align*}
&\lim_{n\rightarrow\infty}\theta_{\Lambda_n}(n^2t)-\det(\Lambda_n)e^{-2dn^2t}I_0(2n^2t)^{d-p}\sum_{(k_1,\ldots,k_p)\in\mathbb{Z}^p}\prod_{i=1}^pI_{k_i\alpha_ia_n}(2n^2t)\\
&=\Theta_{\mathbb{R}^{d-p}/B\mathbb{Z}^{d-p}}(t)-\frac{\det(B)}{(4\pi t)^{(d-p)/2}}.
\end{align*}
We have
\begin{align*}
&\theta_{\Lambda_n}(n^2t)-\det(\Lambda_n)e^{-2dn^2t}I_0(2n^2t)^{d-p}\sum_{(k_1,\ldots,k_p)\in\mathbb{Z}^p}\prod_{i=1}^pI_{k_i\alpha_ia_n}(2n^2t)\nonumber\\
&=\Bigg(\sum_{(k_1,\ldots,k_p)\in\mathbb{Z}^p}\prod_{i=1}^p\alpha_ia_ne^{-2n^2t}I_{k_i\alpha_ia_n}(2n^2t)\Bigg)\Bigg(\sum_{\substack{(k_1,\ldots,k_{d-p})\\ \in\mathbb{Z}^{d-p}\backslash\{0\}}}\prod_{i=1}^{d-p}\beta_ine^{-2n^2t}I_{k_i\beta_in}(2n^2t)\Bigg).
\end{align*}
The first product of the above can be bounded using Proposition \ref{prop4.7bis}. Indeed we have that for all $i=1,\ldots,p$ there exists an $n_{i,0}$ such that for all $n\geqslant n_{i,0}$
\begin{equation*}
\alpha_ia_ne^{-2n^2t}\sum_{k_i\in\mathbb{Z}}I_{k_i\alpha_ia_n}(2n^2t)<\frac{3}{2}.
\end{equation*}
The second product can be rewritten in $d-p$ sums with exactly $r$ of the $k_i$ which are non-zero and $d-p-r$ which are zero. Since the $(k_1,\ldots,k_{d-p})=0$ is taken off the sum, we have $1\leqslant r\leqslant d-p$. Let $n_0=\max_{1\leqslant i\leqslant p}{n_{i,0}}$. From inequality (\ref{boundI0}) and Lemma \ref{lemma4.6} we have that for $t>0$ and all $n\geqslant n_0$ the above is less equal than
\begin{align*}
&2^{d-p}\det(B)\sum_{r=1}^{d-p}C^{d-p-r}t^{-(d-p)/2}\prod_{i=1}^r\sum_{k_i=1}^\infty\left(1+\frac{\beta_ik_i}{2n_0t}\right)^{-n_0\beta_ik_i/2}\\
&\leqslant2^{d-p}\det(B)\sum_{r=1}^{d-p}C^{d-p-r}t^{-(d-p)/2}\prod_{i=1}^r\frac{1}{\left(1+\beta_i/(2n_0t)\right)^{n_0\beta_i/2}-1}\\
&\leqslant2^{d-p}\det(B)\sum_{r=1}^{d-p}C^{d-p-r}t^{-(d-p)/2}\left(\prod_{i=1}^r(\beta_i/(2n_0))^{n_0\beta_i/2}\right)t^{rn_0\min_{1\leqslant i\leqslant d-p}{\beta_i}/2}.
\end{align*}
Hence if we choose $n_0=2(d-p)/\min_{1\leqslant i\leqslant d-p}{\beta_i}+1$ the above is integrable on $(0,1)$ with respect to the measure $dt/t$. The proposition then follows from the Lebesgue dominated convergence Theorem.
\end{proof}
We now study the convergence of the integral over $(1,\infty)$. The theta function can be written as the product of two theta functions, that is
\begin{equation*}
\theta_{\Lambda_n}(n^2t)=\theta_{\textrm{diag}(\beta_1n,\ldots,\beta_{d-p}n)}(n^2t)\theta_{\textrm{diag}(\alpha_1a_n,\ldots,\alpha_pa_n)}(n^2t).
\end{equation*}
The first theta function can be bounded using Lemma $5.3$ in \cite{chinta2010zeta} that we recall below.
\begin{lemma}
\label{lemma5.3}
Let
\begin{equation*}
\theta_{\textnormal{abs}}(t)=2\sum_{j=1}^\infty e^{-cj^2t}
\end{equation*}
with $c=4\pi^2(1-\pi^2/24)^2$. Let $n_0$ be a positive integer. Then for any $t>0$ and $n\geqslant n_0$ we have the bound
\begin{equation*}
\theta_{\textnormal{diag}(\beta_1n,\ldots,\beta_{d-p}n)}(n^2t)\leqslant\prod_{i=1}^{d-p}\left(1+e^{-4n_0^2t}+\theta_\textnormal{abs}(t/(4\beta_i^2))\right).
\end{equation*}
\end{lemma}
It is easy to verify that similarly the second theta function can be bounded by the following
\begin{equation}
\label{boundIa}
\theta_{\textrm{diag}(\alpha_1a_n,\ldots,\alpha_pa_n)}(n^2t)\leqslant\left(1+e^{-4t}+\theta_\textrm{abs}(t)\right)^p.
\end{equation}
Therefore it follows that $\theta_{\Lambda_n}(n^2t)-1$ is $dt/t$-integrable on $(1,\infty)$. So by the Lebesgue dominated convergence Theorem we can exchange the limit and integral. Hence the following proposition is proved:
\begin{proposition}
With the above notation we have that
\label{convTheta-1}
\begin{equation*}
\lim_{n\rightarrow\infty}\int_1^\infty\left(\theta_{\Lambda_n}(n^2t)-1\right)\frac{dt}{t}=\int_1^\infty\left(\Theta_{\mathbb{R}^{d-p}/B\mathbb{Z}^{d-p}}(t)-1\right)\frac{dt}{t}.
\end{equation*}
\end{proposition}
\begin{proposition}
\label{convV}
With the above notation we have that
\begin{equation*}
\lim_{n\rightarrow\infty}\int_1^\infty \det(\Lambda_n)e^{-2dn^2t}I_0(2n^2t)^{d-p}\sum_{(k_1,\ldots,k_p)\in\mathbb{Z}^p}\prod_{i=1}^pI_{k_i\alpha_ia_n}(2n^2t)\frac{dt}{t}=\frac{2}{d-p}\frac{\det(B)}{(4\pi)^{(d-p)/2}}.
\end{equation*}
\end{proposition}
\begin{proof}
Combining (\ref{boundI0}) with (\ref{boundIa}) we have
\begin{equation*}
\det(\Lambda_n)e^{-2dn^2t}I_0(2n^2t)^{d-p}\sum_{(k_1,\ldots,k_p)\in\mathbb{Z}^p}\prod_{i=1}^pI_{k_i\alpha_ia_n}(2n^2t)\leqslant Ct^{-(d-p)/2}(1+e^{-4t}+\theta_\textrm{abs}(t))^p
\end{equation*}
for some constant $C>0$, which is $dt/t$-integrable on $(1,\infty)$. The result follows from the pointwise convergence and from the Lebesgue dominated convergence Theorem.
\end{proof}
Since $\int_1^\infty e^{-n^2t}dt/t\rightarrow0$ as $n\rightarrow\infty$, the asymptotic of the $\mathcal{H}_{\Lambda_n}$ term then follows from Lemma \ref{int01exp}, Propositions \ref{int01Theta}, \ref{convTheta-1} and \ref{convV}:
\begin{align*}
\mathcal{H}_{\Lambda_n}&=2\log n-\int_0^1\left(\Theta_{\mathbb{R}^{d-p}/B\mathbb{Z}^{d-p}}(t)-\frac{\textrm{det}(B)}{(4\pi t)^{(d-p)/2}}\right)\frac{dt}{t}-\Gamma'(1)\nonumber\\
&\ \ \ -\int_1^\infty\left(\Theta_{\mathbb{R}^{d-p}/B\mathbb{Z}^{d-p}}(t)-1\right)\frac{dt}{t}+\frac{2}{d-p}\frac{\textrm{det}(B)}{(4\pi)^{(d-p)/2}}+o(1)\quad\textrm{as }n\rightarrow\infty.
\end{align*}
Rewriting it in terms of the spectral zeta function with the help of equation (\ref{zeta'(0)}) yields
\begin{equation}
\label{Hterm}
\mathcal{H}_{\Lambda_n}=2\log n-\zeta'_{\mathbb{R}^{d-p}/B\mathbb{Z}^{d-p}}(0)+o(1)\textrm{ as }n\rightarrow\infty.
\end{equation}
The calculation of the lead term in section \ref{ancst} together with equation (\ref{Hterm}) gives Theorem \ref{ThDegTcst}. For the case where $a_n$ grows sublinearly with respect to $n$, the error in the lead term, $(n/a_n)^{d-p}o(1)$, is bigger than the $\mathcal{H}$ term (\ref{Hterm}), therefore the asymptotic is given by
\begin{equation*}
\log\textnormal{det}^\ast\Delta_{\mathbb{Z}^d/\Lambda_n\mathbb{Z}^d}=n^{d-p}a_n^p\det(\Lambda)c_d-\left(\frac{n}{a_n}\right)^{d-p}\left(\det(\Lambda)(4\pi)^{d/2}\Gamma(d/2)\zeta_{\mathbb{R}^p/A^{-1}\mathbb{Z}^p}(d/2)+o(1)\right)
\end{equation*}
as $n\rightarrow\infty$.
\subsection{Examples}
The following examples are here to illustrate the general formula and to highlight the interesting constants appearing in some particular cases. In the examples below, $\alpha_i$ and $\beta_i$ denote non-zero positive integers.
\subsubsection{\texorpdfstring{Example with \boldmath $p=1$ and \boldmath$d=2$\unboldmath\ }{p=1, d=2}}
Let $\Lambda_n=\textrm{diag}(\alpha a_n,\beta n)$ be a sequence of diagonal matrices where $a_n$ grows sublinearly with respect to $n$. In \cite{chinta2010zeta} the authors showed that $c_2=4G/\pi$ where $G$ is the Catalan constant. Then as $n\rightarrow\infty$
\begin{equation*}
\log\textnormal{det}^\ast\Delta_{\mathbb{Z}^2/\Lambda_n\mathbb{Z}^2}=na_n\alpha\beta\frac{4G}{\pi}-\frac{n}{a_n}\left(\frac{\beta}{\alpha}\frac{\pi}{3}+o(1)\right).
\end{equation*}
\subsubsection{\texorpdfstring{Example with \boldmath$p=1$ and any \boldmath$d$\unboldmath\ }{p=1}}
Let $\Lambda_n=\textrm{diag}(\alpha a_n,\beta_1n,\ldots,\beta_{d-1}n)$ be a sequence of diagonal matrices where $a_n$ grows sublinearly with respect to $n$. Then as $n\rightarrow\infty$
\begin{align*}
\log\textnormal{det}^\ast\Delta_{\mathbb{Z}^d/\Lambda_n\mathbb{Z}^d}&=n^{d-1}a_n\textrm{det}(\Lambda)c_d-\left(\frac{n}{a_n}\right)^{d-1}\left(\frac{\beta_1\cdots\beta_{d-1}}{\alpha^{d-1}}\frac{2}{\pi^{d/2}}\Gamma(d/2)\zeta(d)+o(1)\right)
\end{align*}
where $\zeta$ is the Riemann zeta function.
\subsubsection{\texorpdfstring{Example with \boldmath$a_n$ constant and \boldmath$p=d-1$}{p=d-1, an constant}}
Let $\Lambda_n^0=\textrm{diag}(\alpha_1,\ldots,\alpha_{d-1},\beta n)$ be a sequence of diagonal matrices. From (\ref{zeta'}), $-\zeta'_{\mathbb{R}/\beta\mathbb{Z}}(0)=2\log\beta$. Using Proposition \ref{intI0} one has as $n\rightarrow\infty$
\begin{equation*}
\log\textnormal{det}^\ast\Delta_{\mathbb{Z}^d/\Lambda_n^0\mathbb{Z}^d}=n\beta\sum_{j=0}^{\det(A)-1}\argcosh\left(1+\frac{\lambda_j}{2}\right)+2\log n+2\log\beta+o(1)
\end{equation*}
where
\begin{equation*}
\{\lambda_j\}_j=\{2(d-1)-2\sum_{i=1}^{d-1}\cos(2\pi j_i/\alpha_i):j_i=0,1,\ldots,\alpha_i-1,\textnormal{ for }i=1,\ldots,d-1\},
\end{equation*}
$j=0,1,\ldots,\det(A)-1$, are the eigenvalues of the Laplacian on $\mathbb{Z}^{d-1}/A\mathbb{Z}^{d-1}$.
\subsubsection{\texorpdfstring{Example with \boldmath$a_n$ constant, \boldmath$p=1$ and \boldmath$d=3$}{p=1, d=3, an constant}}
Let $\Lambda_n^0=\textrm{diag}(\alpha,\beta_1n,\beta_2n)$ be a sequence of diagonal matrices. From section $6.3$ in \cite{chinta2010zeta}, we have that
\begin{equation*}
-\zeta'_{\mathbb{R}^2/\textrm{diag}(\beta_1,\beta_2)\mathbb{Z}^2}(0)=2\log(\beta_2\eta(i\beta_2/\beta_1)^2)
\end{equation*}
where $\eta$ is the Dedekind eta function defined for $z\in\mathbb{C}$ with $\textrm{Im}(z)>0$ by
\begin{equation*}
\eta(z)=e^{\pi iz/12}\prod_{n=1}^\infty(1-e^{2\pi inz}).
\end{equation*}
Hence as $n\rightarrow\infty$
\begin{align*}
\log\textnormal{det}^\ast\Delta_{\mathbb{Z}^3/\Lambda_n^0\mathbb{Z}^3}&=n^2\beta_1\beta_2\sum_{j=0}^{\alpha-1}\int_0^\infty\left(e^{-t}-I_0(2t)^2e^{-(6-2\cos(2\pi j/\alpha))t}\right)\frac{dt}{t}\\
&\ \ \ +2\log n+2\log(\beta_2\eta(i\beta_2/\beta_1)^2)+o(1).
\end{align*}
Using the special value of $\eta$ at $z=i$, $\eta(i)=\Gamma(1/4)/(2\pi^{3/4})$, one has for the special case $\beta_1=\beta_2=:\beta$ the asymptotic behaviour as $n\rightarrow\infty$
\begin{align*}
\log\textnormal{det}^\ast\Delta_{\mathbb{Z}^3/\Lambda_n^0\mathbb{Z}^3}&=n^2\beta_1\beta_2\sum_{j=0}^{\alpha-1}\int_0^\infty\left(e^{-t}-I_0(2t)^2e^{-(6-2\cos(2\pi j/\alpha))t}\right)\frac{dt}{t}\\
&\ \ \ +2\log n+\log(\beta^2\Gamma(1/4)^4/(16\pi^3))+o(1).
\end{align*}
\section[A comment on circulant graphs with non-fixed generators]{A comment on circulant graphs with non-fixed generators\footnote{At the time of reviewing this paper, this conjecture has been proved and will appear in a forthcoming paper.}}
In \cite{golin2010asymptotic,zhang2005chebyshev} the authors considered circulant graphs with non-fixed generators. In \cite{golin2010asymptotic} they computed the lead term of the asymptotic number of spanning trees. It is conceivable that the techniques used here could be extended to improve their result and compute the second term. In \cite{zhang2005chebyshev} they computed the exact number of spanning trees in $C_{\beta n}^{1,n}$ for $\beta\in\{2,3,4,6,12\}$ via Chebyshev polynomials, but were not able to generalize to other values of $\beta$. We propose a conjecture for the case $\beta=5$:\\
For all $n\geqslant2$,
\begin{align*}
\tau(C_{5n}^{1,n})&=\frac{n}{5}\left(\left(\frac{9-\sqrt{5}+\sqrt{70-18\sqrt{5}}}{4}\right)^n+\left(\frac{9-\sqrt{5}+\sqrt{70-18\sqrt{5}}}{4}\right)^{-n}+\frac{1-\sqrt{5}}{2}\right)^2\\
&\ \ \ \times\left(\left(\frac{9+\sqrt{5}+\sqrt{70+18\sqrt{5}}}{4}\right)^n+\left(\frac{9+\sqrt{5}+\sqrt{70+18\sqrt{5}}}{4}\right)^{-n}+\frac{1+\sqrt{5}}{2}\right)^2.
\end{align*}
Notice that the coefficients in the formula can be expressed in terms of integrals involving modified $I$-Bessel function. Indeed, let
\begin{equation*}
J_k^\beta=\int_0^\infty\left(e^{-t}-e^{-2t(2-\cos(2\pi k/\beta))}I_0(2t)\right)\frac{dt}{t},\quad k=1,\ldots,\beta-1.
\end{equation*}
Then from Proposition \ref{intI0}, the above can be rewritten as
\begin{align*}
\tau(C_{5n}^{1,n})=\frac{n}{5}&\left(e^{nJ_1^5}+e^{-nJ_1^5}+\frac{1}{2}(1-\sqrt{5})\right)\left(e^{nJ_2^5}+e^{-nJ_2^5}+\frac{1}{2}(1+\sqrt{5})\right)\\
&\times\left(e^{nJ_3^5}+e^{-nJ_3^5}+\frac{1}{2}(1+\sqrt{5})\right)\left(e^{nJ_4^5}+e^{-nJ_4^5}+\frac{1}{2}(1-\sqrt{5})\right).
\end{align*}
Therefore for other values of $\beta$ the general formula might have the form
\begin{equation*}
\tau(C_{\beta n}^{1,n})=\frac{n}{\beta}\prod_{k=1}^{\beta-1}\left(e^{nJ_k^\beta}+e^{-nJ_k^\beta}+\alpha_k^\beta\right),\textrm{ for all }n\geqslant1,
\end{equation*}
where $\alpha_k^\beta$ are coefficients which are not known for $\beta\geqslant7$.

\nocite{*}
\bibliographystyle{plain}
\bibliography{bibliography2}

\end{document}